\documentclass[10pt]{amsart}

\usepackage{amsfonts,amssymb,amsmath,amscd,amstext}
\usepackage[utf8]{inputenc}
\usepackage{microtype}
\usepackage{graphicx}
\usepackage[colorlinks=true,linkcolor=blue,citecolor=blue]{hyperref}
\usepackage{tikz}
\usetikzlibrary{backgrounds}
\usetikzlibrary{arrows}
\usetikzlibrary{shapes,shapes.geometric,shapes.misc}

\tikzstyle{tikzfig}=[baseline=0em,scale=1]

\pgfkeys{/tikz/tikzit fill/.initial=0}
\pgfkeys{/tikz/tikzit draw/.initial=0}
\pgfkeys{/tikz/tikzit shape/.initial=0}
\pgfkeys{/tikz/tikzit category/.initial=0}

\pgfdeclarelayer{edgelayer}
\pgfdeclarelayer{nodelayer}
\pgfsetlayers{background,edgelayer,nodelayer,main}

\tikzstyle{none}=[inner sep=0mm]

\newcommand{\tikzfig}[1]{%
{\tikzstyle{every picture}=[tikzfig]
\IfFileExists{#1.tikz}
  {\input{#1.tikz}}
  {%
    \IfFileExists{./figures/#1.tikz}
      {\input{./figures/#1.tikz}}
      {\tikz[baseline=-0.5em]{\node[draw=red,font=\color{red},fill=red!10!white] {\textit{#1}};}}%
  }}%
}

\tikzstyle{every loop}=[]

\tikzstyle{black dot}=[fill=black, draw=black, shape=circle, scale=0.3]

\tikzstyle{arrow}=[->]
\tikzstyle{dashed_line}=[-, dashed]

\graphicspath{{pictures/}}

\renewcommand{\leq}{\leqslant}
\renewcommand{\geq}{\geqslant}
\renewcommand{\le}{\leqslant}
\renewcommand{\ge}{\geqslant}
\newcommand{\ptl}{\partial}

\newcommand{\rr}{{\mathbb{R}}} 

\newcommand{\bb}{{\mathbb{B}}} 
\newcommand{\hh}{{\mathbb{H}}}
\newcommand{\nn}{{\mathbb{N}}} 
\newcommand{\sph}{{\mathbb{S}}}
\newcommand{\hhh}{{\mathcal{H}}}
\newcommand{\la}{\lambda}

\newcommand{\sg}{\sigma}
\newcommand{\Om}{\Omega}
\newcommand{\eps}{\varepsilon}

\newcommand{\ga}{\gamma}
\newcommand{\Ga}{\Gamma}

\newcommand{\escpr}[1]{\langle#1\rangle}
\newcommand{\norm}[1]{|| #1 ||}

\DeclareMathOperator{\divv}{div}
\DeclareMathOperator{\intt}{int}

\newtheorem{theorem}{Theorem}[section]
\newtheorem{proposition}[theorem]{Proposition}
\newtheorem{lemma}[theorem]{Lemma}
\newtheorem{corollary}[theorem]{Corollary}

\theoremstyle{definition}

\newtheorem{remark}[theorem]{Remark}
\newtheorem{example}[theorem]{Example}

\newtheorem{definition}[theorem]{Definition} 

\theoremstyle{remark}

\numberwithin{equation}{section}

\begin{document}

\title{Pansu-Wulff shapes in $\hh^1$}

\author[J.~Pozuelo]{Juli\'an Pozuelo} 


\author[M.~Ritor\'e]{Manuel Ritor\'e} 

\address{Departamento de Geometr\'{\i}a y Topolog\'{\i}a \& Research Unit MNat \\
Universidad de Granada \\ E--18071 Granada \\ Espa\~na}

\email{pozuelo@ugr.es}
\email{ritore@ugr.es}

\date{\today}

\thanks{The authors have been supported by MEC-Feder grant MTM2017-84851-C2-1-P, Junta de Andalucía grant A-FQM-441-UGR18, MSCA GHAIA, and Research Unit MNat UCE-PP2017-3. This research was also funded by the Consejería de economía, conocimiento, empresas y universidad and European Regional Development Fund (ERDF), ref. SOMM17/6109/UGR} 
 
\begin{abstract} 
We consider an asymmetric left-invariant norm $||\cdot ||_K$ in the first Heisenberg group $\mathbb{H}^1$ induced by a convex body $K\subset\mathbb{R}^2$ containing the origin in its interior. Associated to $\norm{\cdot}_K$ there is a perimeter functional, that coincides with the classical sub-Riemannian perimeter in case $K$ is the closed unit disk centered at the origin of $\rr^2$. Under the assumption that $K$ has $C^2$ boundary with strictly positive geodesic curvature we compute the first variation formula of perimeter for sets with $C^2$ boundary. The localization of the variational formula in the non-singular part of the boundary, composed of  the points where the tangent plane is not horizontal, allows us to define a mean curvature function $H_K$ out of the singular set. In the case of non-vanishing mean curvature, the condition that $H_K$ be constant implies that the non-singular portion of the boundary is foliated by horizontal liftings of translations of $\ptl K$ dilated by a factor of $1/H_K$. Based on this we can define a sphere $\mathbb{S}_K$ with constant mean curvature $1$  by considering the union of all horizontal liftings of $\partial K$ starting from $(0,0,0)$ until they meet again in a point of the vertical axis. We give some geometric properties of this sphere and, moreover, we prove that, up to non-homogeneous dilations and left-translations, they are the only  solutions of the sub-Finsler isoperimetric problem in a restricted class of sets.
\end{abstract}

\maketitle

\thispagestyle{empty}

\bibliographystyle{abbrv} 

\tableofcontents

\section{Introduction}

In this paper we consider critical points of the perimeter associated to an asymmetric sub-Finsler structure in the first Heisenberg group $\hh^1$. Such a structure is defined by means of an asymmetric left-invariant norm $\norm{\cdot}$ on the horizontal distribution $\hhh$ of $\hh^1$, that is referred to in this paper simply as a norm. 
If we fix any basis of left-invariant horizontal vector fields, any left-invariant norm is uniquely determined by a convex body (compact set with non-empty interior) $K\subset\rr^2$ containing $0$ in its interior. We write $\norm{\cdot}_K$ to indicate the dependence of the norm on $K$. The case of a symmetric norm corresponds to a centrally symmetric convex body. The norm associated to the closed unit disc $D$ centered at $0$ coincides with the Euclidean norm and  is denoted by $|\cdot |$.
Symmetric sub-Finsler structures in $\hh^1$ have received intense interest recently, specially the study of geodesics \cite{MR3657277,MR3908389}, see \cite{MR1867362} for the classical sub-Riemannian case, and the associated Minkowski content \cite{snchez2017subfinsler,1711.01585}. General asymmetric sub-Finsler structures have an associated asymmetric distance and might have different metric properties, see \cite{MR3108873,MR3210893} and \cite{MR2977440}.

On $\hh^1$ we always consider the standard basis of left-invariant vector fields 
\[
X=\frac{\ptl}{\ptl x}+y\,\frac{\ptl}{\ptl t}, \quad Y=\frac{\ptl}{\ptl y}-x\,\frac{\ptl}{\ptl t} ,\quad T=\frac{\ptl}{\ptl t},
\]
and the left-invariant Riemannian metric $g$, also denoted by $\escpr{\cdot,\cdot}$, making $X,Y,T$ orthonormal. The associated Riemannian measure is the Haar measure of the group, and coincides with the Lebesgue measure of the Euclidean space $\rr^3$. The measure of a set $E$ is the volume of the set and is denoted by $|E|$. The volume element is denoted by $d\hh^1$. 

Given a left-invariant norm $\norm{\cdot}_K$, a measurable set $E\subset\hh^1$ and an open set $\Om\subset\hh^1$, we define the sub-Finsler perimeter of $E$ in $\Om$ by
\[
|\ptl E|_K(\Om)=\sup\bigg\{\int_E\divv U\,d\hh^1: U\in\mathcal{H}_0^1(\Om), \norm{U}_{K,\infty}\le 1\bigg\},
\]
where $\mathcal{H}_0^1(\Om)$ is the set of $C^ 1$ horizontal vector fields with compact support in $\Om$ and $\norm{\cdot}_{K,\infty}$ is the infinity norm associated to $\norm{\cdot}_K$. The perimeter associated to the Euclidean norm $|\cdot|$ is the sub-Riemannian perimeter as it is defined in \cite{MR1404326,MR1871966,MR1437714}. A set has finite perimeter for a given norm if and only if it has finite perimeter for the standard sub-Riemannian perimeter. Hence all known results in the standard case apply to the sub-Finsler perimeter.

In case the boundary $S$ of $E$ is a $C^1$ or Euclidean lipschitz surface, the perimeter of $E$ is given by the sub-Finsler area functional
\begin{equation}
\label{eq:K-area}
\tag{*}
A_K(S)=\int_S \norm{N_h}_{K,*}\,dS,
\end{equation}
where $\norm{\cdot}_{K,*}$ is the dual norm of $\norm{\cdot}_K$, $N_h$ is the orthogonal projection to the horizontal distribution of the Riemannian unit normal $N$, and  $dS$ is the Riemannian measure on $S$.

If we consider a convex set $K$ with boundary of class $C^2_+$ (i.e., so that $\ptl K$ is of class $C^2$ and $\ptl K$ has positive geodesic curvature everywhere), we may compute the first variation of the area functional associated to a vector field $U$ with compact support in the regular part of $S$ to get
\[
A'_K(0)=\int_S u\,\big(\divv_S\eta_K\big)\,dS.
\]
In this formula $u=\escpr{U,N}$ is the normal component of the variation and $\divv_S\eta_K$ is the divergence on $S$ of the vector field $\eta_K=\pi_K(\nu_h)$, where $\nu_h=N_h/|N_h|$ is the horizontal unit normal and $\pi_K$ is the map projecting any vector $v\neq 0$ to the intersection of the supporting line in the direction of $v$ with $\norm{\cdot}_K=1$ (the boundary of $K$). The strict convexity of $\norm{\cdot}_K$ implies that this map is well-defined. 

The function $H_K=\divv_S\eta_K$ appearing in the first variation of perimeter is called the \emph{mean curvature} of $S$. Further calculations imply that $H_K$ is equal to $\escpr{D_Z\eta_K,Z}$, where $Z=-J(\nu_h)$ is the horizontal direction on the regular part of $S$. Hence the mean curvature function is localized on the horizontal curves of $S$. It is not difficult to check that a horizontal curve in a surface with mean curvature $H_K$ must satisfy a differential equation depending on $H_K$. Hence we can reconstruct the regular part of a surface with prescribed mean curvature by taking solutions of this differential equation. Furthermore, we might be able classify surfaces with prescribed mean curvature by classifying solutions of this ordinary differential equation and by looking at the interaction of these curves with the singular set $S_0$ of $S$ composed of the points where the tangent plane is horizontal, as was done in \cite{MR2435652} for the standard sub-Riemannian perimeter.

Key observations are that horizontal straight lines are solutions of the differential equation for $H_K=0$ and that horizontal liftings of the curve $\norm{\cdot}_K=1$ are solutions for $H_K=1$. The strict convexity of $\norm{\cdot}_K=1$ together with the invariance of the equation by left-translations and dilations imply that all solutions are of this type. 

Hence, given a convex body $K\subset\rr^2$ containing $0$ in its interior and its associated left--invariant norm $\norm{\cdot}_K$, we consider the set $\bb_K$ obtained as the ball enclosed by the horizontal liftings of all translations of the curve $\ptl K$ containing $0$. It is not difficult to prove that this way we obtain a topological sphere $\sph_K$ with two poles on the same vertical line, that is the union of two graphs. Moreover the boundary of $\bb_K$ is $C^2$ outside the poles (indeed $C^\ell$ if the boundary of $K$ is of class $C^\ell$, $\ell\ge 2$) and of regularity $C^2$ around the poles. When $K=D$, these sets were build by P.  Pansu \cite{MR829003} and are frequently referred to as Pansu spheres. We remark that Pansu spheres' $\bb_D$ are of class $C^2$ but not $C^3$ near the singular points, see Proposition~3.15 in \cite{MR2386783} and  Example~3.3 in \cite{MR2435652}.

\begin{figure}[h]
\includegraphics[width=0.55\textwidth]{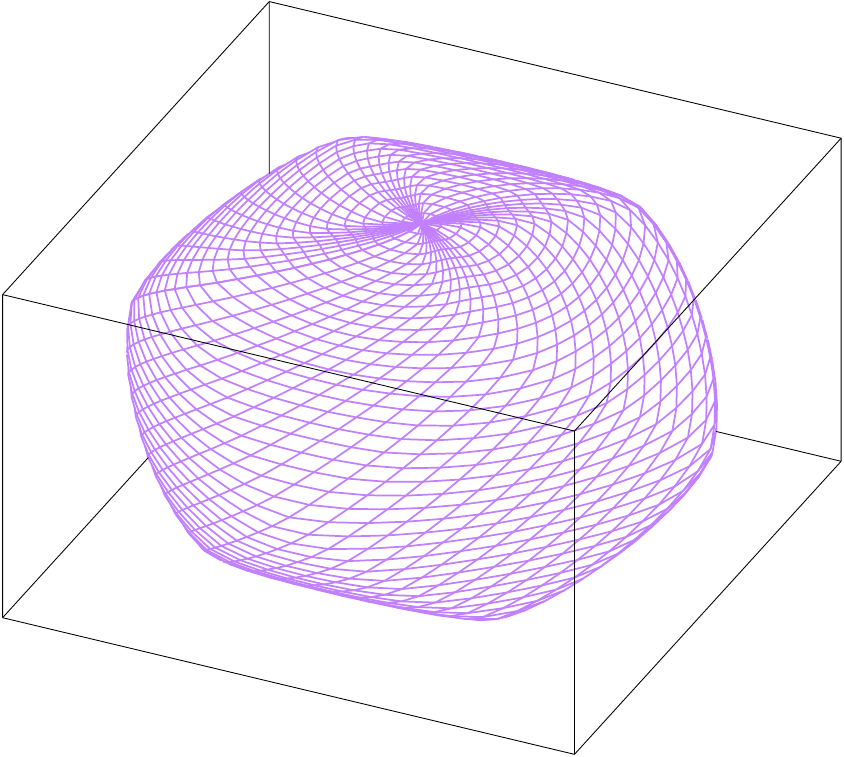}
\caption{The set $\bb_K$ when $K$ is the unit ball of the $r$-norm $\norm{(x,y)}_r=\big(|x|^r+|y|^r\big)^{1/r}$, $r=1.5$}
\end{figure}

\begin{figure}[h]
\includegraphics[width=0.55\textwidth]{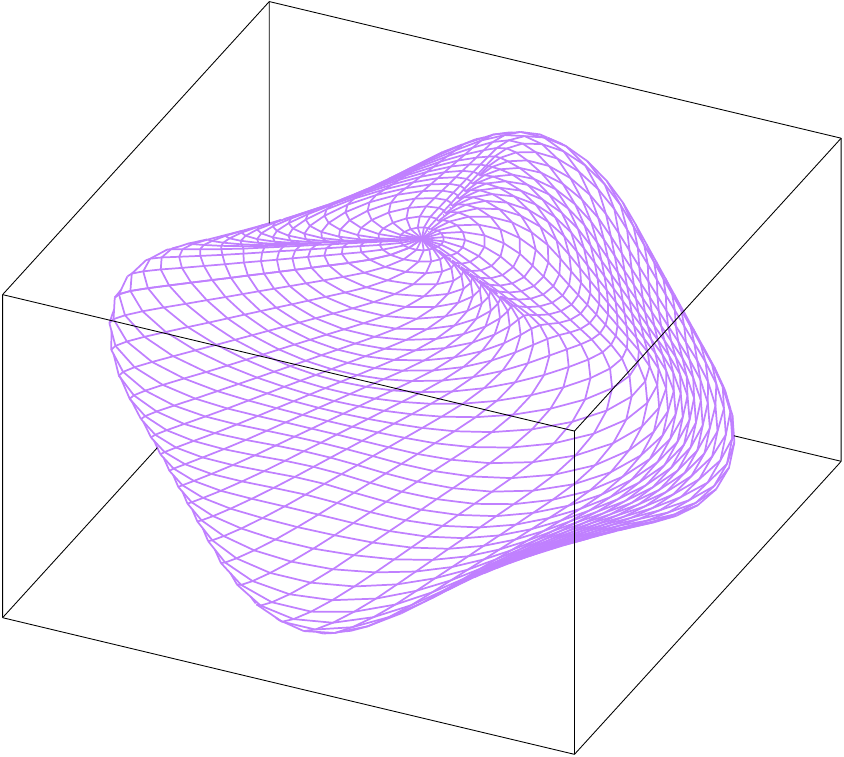}
\caption{The set $\bb_K$ when $K$ is a smooth approximation of the triangular norm}
\end{figure}

We observe that these objects have constant mean curvature. Hence they are critical points of the sub-Finsler area functional under a volume constraint. Further evidence that they have stronger minimization properties is given in Section~\ref{sec:calibration}, where it is proven that, under a geometric condition, a set of finite perimeter $E$ with volume equal to the volume of $\bb_K$ has perimeter larger than or equal to the one of the ball $\bb_K$.
 A slightly weaker result for the Euclidean norm was proven in \cite{MR2898770}.

We have organized this paper into several sections. In the next one we fix notation and  give some background, focusing specially in properties of the sub-Finsler perimeter. In section~\ref{sec:1stvar} we compute the first variation of perimeter for surfaces of class $C^2$ and prove the property that the regular part of the surface is foliated by horizontal liftings of  translations of homothetic expansions of $\ptl K$. In section~\ref{sec:examples} we define the Pansu-Wulff shapes and compute some examples and prove regularity properties of these objects. In Section~\ref{sec:geomprop} we study some geometric properties of the Pansu-Wulff shapes and, finally, in Section~\ref{sec:calibration} we obtain a minimization property of these Pansu-Wulff shapes. This property indicates that these shapes are good candidates to be solutions of the sub-Finsler isoperimetric problem in $\hh^1$.

Some justification on the terminology \emph{Pansu-Wulff shape} must be given. Consider a norm $\norm{\cdot}$ in Euclidean space and its dual norm $\norm{\cdot}_*$. For a Lipschitz surface $S$, the integral
\[
\int_S\norm{N}_*dS,
\]
where $N$ is an a.e. unit normal to $S$, defines a functional that represents the Gibbs free energy, proportional to the area of the surface of contact and to the surface tension, of an anisotropic interface separating two fluids or gases. The contribution of each element of area depends on the orientation. An equilibrium state is obtained by minimizing the free energy  for a drop of given volume. This is an isoperimetric problem in mathematical terms.

The solutions of this problem were described by the crystallographer G. Wulff in 1895: they are translations and dilations of the set $\{x\in\rr: \norm{x}\le 1\}$, usually referred to as the \emph{Wulff shape} of the free energy. A first mathematical proof of this fact was given by Dinghas \cite{zbMATH03045073}. Other versions of Wulff's results were given by Busemann \cite{MR31762}, Taylor \cite{MR493671}, Fonseca \cite{MR1116536} and Fonseca and Müller  \cite{MR1130601}; see also Gardner \cite{MR1898210},  Burago and Zalgaller \cite{MR936419}, Van Schaftingen \cite{MR2245755}, and Figalli, Maggi and Pratelli \cite{MR2672283}.

The counterpart of the free energy in the Heisenberg group $\hh^1$ is given in formula \eqref{eq:K-area}. When $K=D$ we obtain the classical sub-Riemannian area. In his Ph. Thesis, Pansu exhibited in \cite{MR829003} an example of an area-stationary candidate, that coincides with the sub-Finsler  Wulff shape we obtain this paper, and conjectured that this set is a solution of the sub-Riemannian isoperimetric problem in $\hh^1$. While many partial results have been obtained in the direction of proving this conjecture, see \cite{MR2271950,MR2435652,MR2898770,MR2548252,MR3412408,MR2402213,MR1976833,MR2177813,MR2386783,MR3987854} and the monograph \cite{MR2312336}, it still remains open.  

The authors were informed recently of the paper \cite{monti-finsler}, where the same problem is considered.

They would like to thank César Rosales for interesting discussions and Enrico Le Donne for suggesting the terminology \textit{Pansu-Wulff shapes}.

\section{Preliminaries}

\subsection{The Heisenberg group $\hh^1$}

 The Lie group $(\rr^3,*)$, where $*$ is the product defined, for any pair of points $[z,t]$, $[z',t']\in\rr^3\equiv\mathbb{C}\times\rr$, by
\[
[z,t]*[z',t']:=[z+z',t+t'+\text{Im}(z\overline{z}')], \qquad (z=x+iy).
\]
is referred to as the \emph{first Heisenberg group} and denoted by $\hh^1$. For $p\in\hh^1$, the \emph{left translation} by $p$ is the diffeomorphism $L_p(q)=p*q$.  A basis of left invariant vector fields is given by
\begin{equation*}
X:=\frac{\ptl}{\ptl x}+y\,\frac{\ptl}{\ptl t}, \qquad
Y:=\frac{\ptl}{\ptl y}-x\,\frac{\ptl}{\ptl t}, \qquad
T:=\frac{\ptl}{\ptl t}.
\end{equation*}
The \emph{horizontal distribution} $\mathcal{H}$ in $\hh^1$ is the smooth planar distribution generated by $X$ and $Y$.  The \emph{horizontal projection} of a vector $U$ onto $\mathcal{H}$ will be denoted by $U_{h}$.  A vector field $U$ is called \emph{horizontal} if $U=U_h$. A \emph{horizontal curve} is a $C^1$ curve whose tangent vector lies
in the horizontal distribution.

We denote by $[U,V]$ the Lie bracket of two $C^1$ vector fields $U$, $V$ on $\hh^1$. Note that $[X,T]=[Y,T]=0$, while $[X,Y]=-2T$.  The last equality implies that $\mathcal{H}$ is a
bracket generating distribution. Moreover, by Frobenius Theorem we have that $\mathcal{H}$ is nonintegrable.  The vector fields $X$ and $Y$ generate the kernel of the (contact) $1$-form
$\omega:=-y\,dx+x\,dy+dt$.

We shall consider on $\hh^1$ the (left invariant) Riemannian metric $g=\escpr{\cdot\,,\cdot}$ so that $\{X,Y,T\}$ is an orthonormal basis at every point, and the associated Levi-Civit\'a connection $D$.  The modulus of a vector field $U$ with respect to this Riemannian metric will be denoted by $|U|$.  The following derivatives can be easily computed 
\begin{alignat}{2}
\notag D_{X}X&=0, \qquad \ \ \ \, D_{Y}Y=0, \qquad \,D_{T}T=0, \\
\label{eq:christoffel}
D_{X}Y&=-T, \qquad \, D_{X}T=Y, \qquad \, D_{Y}T=-X, \\
\notag D_{Y}X&=T, \qquad \ \ \,\,D_{T}X=Y, \qquad D_{T}Y=-X.
\end{alignat}

For any vector field $U$ on $\hh^1$ we define $J(U)=D_UT$.  Then we have $J(X)=Y$, $J(Y)=-X$ and $J(T)=0$, so that $J^2=-\text{Identity}$ when restricted to the horizontal distribution.  

We consider the first Heisenberg group $\hh^1$, and refer to \cite{MR2435652} for notation and background.

\subsection{The pseudo-hermitian connection}

The \emph{pseudo-hermitian} connection $\nabla$ on $\hh^1$ is the only affine connection satisfying the following properties:
\begin{enumerate}
	\item $\nabla$ is a metric connection,
	\item $\text{Tor}(U,V)=2\,\escpr{J(U),V}T$ for all vector fields $U,V$.
\end{enumerate}	
The existence of the pseudo-hermitian connection can be easily obtained adapting the proof of existence of the Levi-Civita connection, see Theorem 3.6 in \cite{MR1138207}. 


 We shall use the following relation between the pseudo-hermitian and the Levi-Civita connections.
 \begin{lemma}
 	Let $U$, $V$ and $W$ be vector fields where $V$ and $W$ are horizontal. Then the following equation holds
 	\begin{equation}\label{relcon}
 	\langle \nabla_U V, W\rangle=\langle D_U V, W\rangle+\langle J(W),V\rangle \langle T,U\rangle.
 	\end{equation}
In particular
\begin{equation}
\label{eq:relcon}
\nabla_UV=D_UV-\escpr{T,U}J(V).
\end{equation}
 \end{lemma}	
\begin{proof}
By Koszul formula, see \S~3 in \cite{MR1138207}.
The terms in the first two lines are equal to $\escpr{D_UV,W}$. The last three terms can be computed using the expression for the torsion to get
\begin{equation*}
\escpr{ J(W),V}\escpr{T,U}.
\end{equation*}
This proves \eqref{relcon}. 
\end{proof}

Using Koszul formula it can be easily seen that $\nabla X=\nabla Y=0$.
 
\begin{corollary}
\label{cor:conrel}
Let $\ga:I\to S$ be a curve on $\hh^1$ and let $\nabla/ds$, $D/ds$ be the covariant derivatives induced by the pseudo-hermitian connection and the Levi-Civita connection in $\ga$, respectively. Let $V$ be a vector field along $\ga$. Then we have
\begin{equation}
\label{eq:relcovder}
\frac{\nabla}{ds}V=\frac{D}{ds}V-\escpr{\dot{\ga},T}J(V).
\end{equation}
In particular, if $\ga$ is a horizontal curve, the covariant derivatives coincide.
\end{corollary}

\subsection{Sub-Finsler norms}

The notion of norm we use in these notes is the one of asymmetric norm. This is a non-negative function $\norm{\cdot}:V\to\rr$ defined on a finite-dimensional real vector space $V$ satisfying
\begin{enumerate}
\item $\norm{v}=0$ if and and only if $v=0$,
\item $\norm{\la v}=\la\norm{v}$, for all $\la\ge 0$ and $v\in V$, and
\item $\norm{v+w}\le\norm{v}+\norm{w}$, for all $v,w\in V$.
\end{enumerate}
We stress the fact that we are not assuming the symmetry property $\norm{-v}=\norm{v}$.

Associated to a given a norm $||\cdot||$ in $V$ we have the set $F=\{u\in V:||u||\le 1\}$, which is compact, convex and includes $0$ in its interior. Reciprocally, given a compact convex set $K$ with $0\in\intt(K)$, the function $||u||_K=\inf\{\la\ge 0:  u\in\la K\}$ defines a norm in $V$ so that $F=\{u\in V:||u||_K\le 1\}$. The set $F$ is referred to as the closed unit ball (centered at $0$) of the norm $||\cdot||$. 

Given a norm $\norm{\cdot}$ and an scalar product $\escpr{\cdot,\cdot}$ in $V$, we consider its dual norm $\norm{\cdot}_*$ of $||\cdot||$ with respect to $\escpr{\cdot,\cdot}$ defined by
\[
\norm{u}_*=\sup_{\norm{v}\le 1}\escpr{u,v}.
\]


The dual norm is the support function $h$ of the unit ball $K=\{u\in V: \norm{u}\le 1\}$ with respect to the scalar product $\escpr{\cdot,\cdot}$. From this point on, we assume that $\norm{\cdot}$ is smooth (i.e., it is $C^\infty$ in $V\setminus\{0\}$) and strictly convex:
\[
\norm{\la u+(1-\la)v}<1,\quad\text{for all }\la\in (0,1), \text{when }  u\neq v, \norm{u}=\norm{v}=1.
\]

Given $u\in V$, the compactness of the unit ball of $\norm{\cdot}$ and the continuity of $\norm{\cdot}$ implies the existence of $u_0\in V$ satisfying equality $\norm{u}_*=\escpr{u,u_0}$. Moreover, it can be easily checked that $\norm{u_0}=1$. In general, a point $u_0$ satisfying this property is not unique, but uniqueness follows from the assumption that $\norm{\cdot}$ is strictly convex: this is proved by contradiction assuming the existence of another point $u_0'$ with $\norm{u_0'}\le 1$ satisfying $\norm{u}_*=\escpr{u,u_0'}$. Of course $u_0'$ must also satisfy $\norm{u_0'}=1$. Then all the points $v$ in the segment $[u_0,u_0']$ satisfy $\norm{v}\le 1$ and $\norm{u}_*=\escpr{u,v}$; hence $\norm{v}=1$. But this contradicts the strict convexity of $\norm{\cdot}$ unless $u_0=u_0'$. We shall define $\pi(u)$ as the only vector satisfying $\norm{\pi(u)}=1$ and
\[
h(u)=\norm{u}_*=\escpr{u,\pi(u)}.
\]
If $\la>0$ then it is easily checked that $\pi(\la u)=\pi(u)$.

We further assume that $K$ is of class $C^\ell_+$, with $\ell\ge 2$. This means that $\ptl K$ is of class $C^\ell$, $\ell\ge 2$, and that the geodesic curvature of $\ptl K$ is everywhere positive. Hence the Gauss map $N:\ptl K\to\sph^1$ to the unit circle is a diffeomorphism of class $C^{\ell-1}$. Since $\pi=N^{-1}$ we conclude that $\pi$ is of class $C^{\ell-1}$.  Moreover, by Corollary~1.7.3 in \cite{MR3155183} we have
\begin{equation*}
\nabla h(u)=N^{-1}\bigg(\frac{u}{|u|}\bigg),
\end{equation*}
and so $h$ is of class $C^\ell$.

Given a norm $\norm{\cdot}_0$ in $\hhh_0$, we extend it by left-invariance to a norm $\norm{\cdot}$ in the whole horizontal distribution $\hhh$ by means of the formula
\begin{equation}
\norm{v}_p=\norm{d\ell_p^{-1}(v)}_0, \qquad  p\in\hh^1, v\in\hhh_p.
\end{equation}
In particular, for a horizontal vector field $fX+gY$, its norm at a point $p\in\hh^1$ is given by $\norm{f(p)X_0+g(p)Y_0}_0$.
Identifying the vector $aX_0+bY_0\in \hhh_0$ with the Euclidean vector $(a,b)$, we can define a norm in $\rr^2$ by the formula $\norm{(a,b)}_e=\norm{aX_0+bY_0}_0$.

We consider the norm $(\norm{\cdot}_{0})_*$, dual to $\norm{\cdot}_{0}$ in $\hhh_0$, and we extend it by left-invariance to a norm $\norm{\cdot}_*$ in $\hhh$. It can be easily checked that $(\norm{\cdot}_*)_p$ is the dual norm to $\norm{\cdot}_p$ since
\begin{align*}
(\norm{v}_*)_p=(\norm{d\ell_p^{-1}(v)}_0)_{*}&=\sup_{\norm{w}_0\le 1, w\in\hhh_0}\escpr{d\ell_p^{-1}(v),w}
\\
&=\sup_{\norm{w'}_p\le 1, w'\in\hhh_p}\escpr{v,w'}
\\
&=(\norm{v}_p)_*.
\end{align*}


When $\norm{\cdot}_0$ is $C^l_+$ with $l\geq2$, all norms $\norm{\cdot}_p$ are $C^l_+$. Given a horizontal vector field $U$ of class $C^1$, we define $\pi(U)$ as the $C^1$ horizontal vector field satisfying
\begin{equation}
\label{eq:pi*}
\norm{U}_*=\escpr{U,\pi(U)},
\end{equation}
or, equivalently, $(\norm{U_p}_p)_*=\escpr{U_p,\pi(U)_p}$ for all $p\in\hh^1$.  We recall that $\pi(f U)=\pi(U)$ for any positive smooth function $f$.

\subsection{sub-Finsler perimeter} 
Let $E\subset\hh^1$ be a measurable set, $\norm{\cdot}_K$ the left-invariant norm associated to a convex body $K\subset\rr^2$ so that $0\in\intt(K)$, and $\Om\subset\hh^1$ an open subset. We say that $E$ has locally finite $K$-perimeter in $\Om$ if for any relatively compact open set $V\subset\Om$ we have
\[
|\ptl E|_K(V)=\sup\bigg\{\int_E\divv(U)\,d\hh^1: U\in\hhh_0^1(V), \norm{U}_{K,\infty}\le 1\bigg\}<+\infty.
\]
In this expression, $\hhh_0^1(V)$ is the space of horizontal vector fields of class $C^1$ with compact support in $V$, and $\norm{U}_{K,\infty}=\sup_{p\in V} \norm{U_p}_K$. The integral is computed with respect to the Riemannian measure $d\hh^1$ of this left-invariant metric.

Let $K,K'$ bounded convex bodies containing $0$ in its interior. Then there exist constants $\alpha,\beta>0$ such that
\[
\alpha \norm{x}_{K'}\le \norm{x}_K\le \beta \norm{x}_{K'},\quad \text{for all }x\in\rr^2.
\]
Let $E\subset\hh^1$ be a measurable set, $\Om\subset\hh^1$ an open set and $V\subset\Om$ a relatively open set. Take $U\in\hhh_0^1(V)$ a vector field with $\norm{U}_{K,\infty}\le 1$. Hence $\norm{\alpha U}_{K'}\le \norm{U}_{K}\le 1$ and
\begin{equation*}
\int_E \divv(U)d\hh^1=\frac{1}{\alpha}\int_E \divv(\alpha U)\,d\hh^1\le \frac{1}{\alpha} |\ptl E|_{K'}(V),
\end{equation*}
Taking supremum over the set of $C^1$ horizontal vector fields with compact support in $V$ and $\norm{\cdot}_K\le 1$, we get $|\ptl E|_K(V)\le \tfrac{1}{\alpha}|\ptl E|_{K'}(V)$. In a similar way we get the inequality $\tfrac{1}{\beta}|\ptl E|_{K'}(V)\le |\ptl E|_K(V)$, so that we have
\begin{equation}
\label{eq:abscont}
\tfrac{1}{\beta} |\ptl E|_{K'}(V)\leq |\ptl E_K|(V)\leq \tfrac{1}{\alpha} |\ptl E|_{K'}(V).
\end{equation}
As a consequence,  $E$ has locally finite $K$-perimeter if and only if it has locally finite $K'$-perimeter. 

Let $E\subset\hh^1$ be a set with locally finite $K$-perimeter in $\Om$. Given the standard basis $X,Y$ of the horizontal distribution, we can define a linear functional $L:C_0^1(\Om,\rr^2)\to \rr$
 by
\[
L(g)=L((g_1,g_2))=\int_E \divv(g_1X+g_2Y)\,d\hh^1.
\]
For any relatively compact open set $V\subset\Om$ we have
\[
C(V):=\sup\{L(g): g\in C_0^1(V,\rr^2), \norm{g}_{K,\infty}\le 1\}<+\infty,
\]

We fix any compact subset $C\subset \Om$ and take a relatively compact open set $V$ such that $C\subset V\subset\Om$. For each $g\in C_0(\Om,\rr^2)$ with support in $K$ we can find a sequence of $C^1$ functions $(g_i)_{i\in\nn}$ with support in $V$ such that $g_i$ converges uniformly to $g$. Hence equality
\[
\overline{L}(g)=\lim_{i\to\infty} L(g_i)
\]
allows to extend $L$ to a linear functional $\overline{L}:C_0(\Om,\rr^2)\to\rr$ satisfying 
\[
\sup\{\overline{L}(g): g\in C_0(\Om,\rr^2), \text{supp}(g)\subset C, \norm{g}_{K,\infty}\le 1\}\leq C(V)<+\infty.
\]

The proof of the Riesz Representation Theorem, see \S~1.8 in \cite{MR3409135}, can be adapted to obtain the existence of a Radon measure $\mu_K$ on $\Om$ and a $\mu_K$-measurable horizontal vector field $\nu_K$ in $\Om$ so that $\nu_K=\nu_1X+\nu_2Y$, with $(\nu_1,\nu_2):\Om\to\rr^2$ a $\mu_K$-measurable function, satisfying
\[
\overline{L}(g)=\int_\Om \escpr{g_1X+g_2Y,\nu_K}\,d\mu_K.
\]
The measure $\mu_K$ is the total variation measure
\[
\mu_K(V)=\sup\{\bar{L}(g):g\in C_0(\Om,\rr^2), \text{supp}(g)\subset V, \norm{g}_{K,\infty}\le 1\}
\]
that coincides with $|\ptl E|_K(V)$ because $\overline{L}$ is a continuous extension of $L$. Henceforth we denote $\mu_K$ by $|\ptl E|_K$.

Let us check that
\begin{equation}
\label{eq:normKnuK=1}
\norm{(\nu_K)_p}_{K,*}=1 \text{ for }|\ptl E|_K\text{-}a.e.\ p.
\end{equation}
Here $\norm{\cdot}_{K,*}$ is the dual norm of $\norm{\cdot}_K$. To prove \eqref{eq:normKnuK=1} we take a relatively compact open set $V\subset\Om$ and $g\in C_0(\Om,\rr^2)$ with $\text{supp}(g)\subset V$ and $\norm{g}_{K,\infty}\le 1$. Since $\escpr{g_1X+g_2Y,\nu_K}\le \norm{\nu_K}_{K,*}$ we have
\[
\overline{L}(g)\le\int_V \norm{\nu_K}_{K,*}d|\partial E|_K.
\]
Taking supremum over such $g$ we have
\[
|\partial E|_K(V)\le\int_V\norm{\nu_K}_{K,*}d|\partial E|_K.
\]
On the other hand, we can take a sequence of functions $(h_i)=((h_1)_i,(h_2)_i)$ with support in $V$ such that $\norm{h_i}_K\le 1$ and $\escpr{(h_1)_iX+(h_2)_iY,\nu_K}$ converges to $\norm{\nu_K}_{K,*}$ $|\partial E|_K$-a.e. This is a consequence of Lusin's Theorem, see \S~1.2 in \cite{MR3409135}, and follows by approximating the measurable function $\pi_K(\nu_K)$ by continuous uniformly bounded functions. Then we would have
\[
\int_V\norm{\nu_K}_{K,*}d|\partial E|_K=\lim_{i\to\infty} \escpr{(h_1)_iX+(h_2)_iY,\nu_K}d|\partial E|_K\le |\partial E|_K(V).
\]
So we would have
\[
|\partial E|_K(V)=\int_V\norm{\nu_K}_{K,*}d|\partial E|_K
\]
and so $\norm{\nu_K}_{K,*}=1$ for $|\partial E|_K$-a.e.

Given two convex sets $K,K'\subset \rr^2$ containing $0$ in their interiors, we shall obtain the following representation formula for the sub-finsler perimeter measure $|\partial E|_K$ and the vector field $\nu_K$
\begin{equation}
\label{eq:rnfinal}
|\partial E|_K=\norm{\nu_{K'}}_{K,*}|\partial E|_{K'},\quad \nu_{K}=\frac{\nu_{K'}}{\norm{\nu_{K'}}_{K,*}}.
\end{equation}
From \eqref{eq:abscont}, there exist two positive constants $\la,\Lambda$ such that
\[
\la|\partial E|_K\le|\partial E|_{K'}\le \Lambda|\partial E|_K.
\]
This implies that each of the Radon measures $|\partial E|_K,|\partial E|_{K'}$ is absolutely continuous with respect to the other one. Hence both Radon-Nikodym derivatives exist. Take a relatively compact open set $V\subset\Om$ and $U\in\hhh_0^1(V)$. Then we have
\begin{equation}
\label{eq:rn-0}
\begin{split}
\int_V\escpr{U,\nu_{K'}}\,d|\partial E|_{K'}&=
\int_V\chi_E \divv(U)\,d\hh^1
\\
&=\int_V\escpr{U,\nu_{K}}\,d|\partial E|_{K}=\int_V\escpr{U,\frac{d|\partial E|_{K}}{d|\partial E|_{K'}}\,\nu_{K}}\,d|\partial E|_{K'}.
\end{split}
\end{equation}
By the uniqueness of $\nu_{K'}$ we have
\begin{equation}
\label{eq:rn-1}
\nu_{K'}=\frac{d|\partial E|_{K}}{d|\partial E|_{K'}}\,\nu_{K},\quad |\partial E|_{K'}\text{-}a.e.
\end{equation}
On the other hand, inserting $U\in\hhh_0^1(V)$ in \eqref{eq:rn-0} with $\norm{U}_K\le 1$ we get
\[
\int_V\escpr{U,\nu_K}d|\partial E|_K=\int_V\escpr{U,\nu_{K'}}\,d|\partial E|_{K'}\le \int_V\norm{\nu_{K'}}_{K,*}d|\partial E|_{K'}.
\]
Taking supremum over $U$ we obtain
\[
\int_V\frac{d|\partial E|_K}{d|\partial E|_{K'}}d|\partial E|_{K'}=|\partial E|_K(V)\le \int_V\norm{\nu_{K'}}_{K,*}d|\partial E|_{K'}
\]
and, since $V$ is arbitrary, we have
\begin{equation}
\label{eq:rn-2}
\frac{d|\partial E|_K}{d|\partial E|_{K'}}\le\norm{\nu_{K'}}_{K,*}\quad |\partial E|_{K}\text{-}a.e. 
\end{equation}

Substituting \eqref{eq:rn-1} into \eqref{eq:rn-2} we have
\[
\frac{d|\partial E|_K}{d|\partial E|_{K'}}\le\norm{\nu_{K'}}_{K,*} =\frac{d|\partial E|_{K}}{d|\partial E|_{K'}} \quad |\partial E|_K\text{-}a.e.
\]
Hence we have equality and so
\begin{equation}
\label{eq:rneq}
\frac{d|\partial E|_K}{d|\partial E|_{K'}}=\norm{\nu_{K'}}_{K,*}\quad |\partial E|_K \text{-}a.e.
\end{equation}
Hence we get from equation \eqref{eq:rnfinal} from \eqref{eq:rneq} and \eqref{eq:rn-1}.

In the case of a set $E$ with $C^1$ boundary $S=\ptl E$ it is not difficult to check that
\[
|\partial E|_K=\norm{N_h}_{K,*}dS,\quad \nu_K=\frac{N_h}{\norm{N_h}_{K,*}},
\]
where $N_h$ is the horizontal projection of the unit normal to $S$ and $dS$ is the Riemannian measure on $S$. Indeed, for the closed unit disk $D\subset\rr^2$ centered at $0$ we know that in the $C^1$ case $\nu_D=\nu_h$ and $|N_h|=\norm{N_h}_{D,*}$. Hence we have
\begin{equation}
\label{eq:muKmuD}
|\partial E|_K=\norm{\nu_h}_{K,*}d|\partial E|_D, \quad \nu_K=\frac{\nu_h}{\norm{\nu_h}_{K,*}}.
\end{equation}
Here $|\partial E|_D$ is the standard sub-Riemannian measure.

\begin{remark}
Some other notions of perimeter and area for higher codimensional submanifolds have been considered in \cite{MR2313532,MR2414951,MR4118581}.
\end{remark}

\subsection{Immersed surfaces in $\hh^1$}
We consider oriented surfaces of class $C^2$ immersed in $\hh^1$ and we shall choose a unit normal to $S$. In case $S$ is the boundary of a domain $\Om\subset\hh^1$, we always choose the \emph{outer} unit normal. The \emph{singular set} of $S$ is denoted by $S_0$ and it is composed of the points in $p\in S$ where the tangent space $T_pS$ coincides with the horizontal  distribution $\hhh_p$. The \emph{horizontal} unit normal $\nu_h$ is defined in $S\setminus S_0$ by
\[
\nu_h=\frac{N_h}{|N_h|}.
\]
The vector field $Z$ is defined by
\[
Z=-J(\nu_h).
\]
The vector field $Z$ is defined on $S\setminus S_0$ and it is tangent to $S$ and horizontal. It generates at every point $p\in S\setminus S_0$ the subspace $T_pS\cap\hhh_p$.


\section{First variation of sub-Finsler area} 	
\label{sec:1stvar}
In this section we fix a convex body $K\subset\rr^2$ containing $0$ in its interior with $C^2_+$ boundary and consider its associated left-invariant norm $\norm{\cdot}_K$ in $\hh^1$. Since the convex body is fixed, we drop the subscript along this section.

Let $S$ be an oriented $C^2$ surface immersed in $\hh^1$. Let $U$ be a $C^2$ vector field with compact support on $S$, normal component $u=\langle U,N\rangle$ and  associated one-parameter group of diffeomorphisms $\{\varphi_s\}_{s\in\rr}$. In this subsection we compute the first variation of the sub-Finsler area $A(s)=A(\varphi_s(S))$. More precisely
\begin{theorem}
\label{thm:1stvar}
Let $S$ be an oriented $C^2$ surface immersed in $\hh^1$. Let $U$ be a $C^2$ vector field with compact support on $S$, normal component $u=\langle U,N\rangle$ and $\{\varphi_s\}_{s\in\rr}$ the associated one-parameter group of diffeomorphisms. Let $\eta=\pi(\nu_h)$. Then we have 
\begin{equation}
\label{eq:A'(0)}
\begin{split}
\frac{d}{ds}\bigg|_{s=0} A(\varphi_s(S))=\int_S\big(u\divv_S \eta
-2u\escpr{N,T}\escpr{J(&N_h),\eta}\big)\,dS
	\\	&
-\int_S \divv_S\big(u\eta^\top\big)\,dS,
\end{split}
\end{equation}
where $\divv_S$ is the Riemannian divergence in $S$, and the superscript $\top$ indicates the tangent projection to $S$.
\end{theorem}

In the proof of Theorem~\ref{eq:A'(0)}  we shall make use of the following Lemma and its consequences.

%

\begin{lemma}
\label{dernor}
Let $\ga:I\to\hh^1$ be a $C^1$ curve, where $I\subset\rr$ is an open interval, and $V$ a horizontal vector field along $\ga$. We have 
\begin{equation}
\label{eq:DnormU}
\frac{d}{ds}\norm{V}_*=\escpr{\frac{D}{ds} V,\pi(V)}+\escpr{\ga',T_{\ga}}\escpr{V,J(\pi(V))}.
	\end{equation}
\end{lemma}	

\begin{proof}
We fix $s_0\in I$ and let $p=\ga(s_0)$. Assume that $\pi(V(s_0))=aX_p+bY_p$, for some $a,b\in\rr$. Take  the vector field $W(s):=aX_{\ga(s)}+bY_{\ga(s)}$ along $\ga$. It coincides with $\pi(V(s_0))$ when $s=s_0$, and it is the restriction to $\ga$ of the left-invariant vector field $aX+bY$. In particular, $\norm{(aX+bY)_{\ga(s)}}_{\ga(s)}=1$ for all $s\in I$. Hence
\[
\norm{V(s)}_*\ge \escpr{V(s),(aX+bY)_{\ga(s)}}\quad\text{for all }s\in I,
\]
and, since equality holds in the above inequality when $s=s_0$, we have
\begin{align*}
\frac{d}{ds}\bigg|_{s=s_0} \norm{V(s)}_*&=\frac{d}{ds}\bigg|_{s=s_0}\escpr{V(s),(aX+bY)_{\ga(s)}}
\\
&=\escpr{\frac{\nabla}{ds}\bigg|_{s=s_0} V(s),\pi(V(s_0))}
\end{align*}
since
\[
\frac{\nabla}{ds}\bigg|_{s=s_0}(aX+bY)_{\ga(s)}=a\nabla_{\ga'(s_0)}X+b\nabla_{\ga'(s_0)} Y=0.
\]
The result follows from the relation between the covariant derivatives given in Equation \eqref{eq:relcovder}.
\end{proof}

\begin{remark}
In the proof of Lemma \ref{dernor} we have obtained the equality
\[
\frac{d}{ds} \norm{V}_*=\escpr{\frac{\nabla}{ds} V,\pi(V)}
\]
for a horizontal vector field $V$ along a curve $\ga$. Since $\nabla$ is a metric connection, we also have
\[
\frac{d}{ds} \norm{V}_*=\escpr{\frac{\nabla}{ds} V,\pi(V)}+
\escpr{V,\frac{\nabla}{ds} \pi(V)}.
\]
Hence we get
\begin{equation}
\label{eq:keyeq}
\langle  V,\frac{\nabla}{ds}\pi(V)\rangle=0
\end{equation}
for a horizontal vector field $V$ along $\ga$, where $\nabla/ds$ is the covariant derivative induced by the pseudo-hermitian connection on $\ga$.
Taking into account the relation between the Levi-Civita and pseudo-hermitian connections we deduce from \eqref{eq:keyeq} and \eqref{eq:relcovder}
\begin{equation}
\label{eq:keyeq2}
\escpr{V,\frac{D}{ds}\pi(V)-\escpr{\dot{\ga},T_\ga} J(\pi(V))}=0.
\end{equation}
\end{remark}

The following is an easy consequence of Lemma~\ref{dernor}


%

\begin{corollary}
Let $F$ be a vector field tangent to $S$ and $\ga$ an integral curve of $F$. We have 
\begin{equation}
\label{eq:corconnection}
\escpr{\frac{D}{ds}\eta_{\ga},\nu_h}=-\escpr{ F,T}\escpr{\eta,J(\nu_h)}.
\end{equation}
In particular, if $F$ is horizontal,
\begin{equation}
\label{eq:corconnection-2}
\escpr{\frac{D}{ds}\eta_{\ga},\nu_h}=0.
\end{equation}
\end{corollary}

\begin{proof}
We take $V=\nu_{h}$ and we get \eqref{eq:corconnection} from equation \eqref{eq:keyeq2}.
\end{proof}
%
%
%
%

\begin{proof}[Proof of Theorem~\ref{thm:1stvar}]
Standard variational arguments, see the proof of Lemma~4.3 in \cite{MR2435652}, yield
\begin{align*}
A'(0)&=\frac{d}{ds}\bigg|_{s=0}A(\varphi_s(S))
=\int_S\bigg(\frac{d}{ds}\bigg|_{s=0}||(N_s)_h||_*+||N_h||_*\divv_S U\bigg)\,dS,
\end{align*}
where $N_s$ is a smooth choice of unit normal to $\varphi_s(S)$ for small $s$. 
We fix a point $p\in S$ and consider the curve $\ga(s)=\varphi_s(p)$. Lemma~\ref{dernor} now implies 
\[
\frac{d}{ds}\bigg|_{s=0} ||(N_s)_h||_*=\escpr{\frac{D}{ds}\bigg|_{s=0}(N_s)_h,\eta_p}+\escpr{U_p,T_p}\escpr{(N_h)_p,J(\eta_p)},
\]
By the definition of $(N_s)_h$ we also have
\[
\frac{D}{ds}\bigg|_{s=0} (N_s)_h=\frac{D}{ds}\bigg|_{s=0} \big(N_s-\escpr{N_s,T}T\big),
\]
where $N_s$ is the Riemannian unit normal to $\varphi_s(S)$. A well-known lemma in Riemannian geometry implies
\[
\frac{D}{ds}\bigg|_{s=0} N_s=-(\nabla_S u)(p)-A_S(U^\top_p),
\]
where $A_S$ is the Weingarten endomorphism of $S$.  Since $\tfrac{D}{ds}\big|_{s=0} T=J(U_p)$ and $\eta$ is horizontal, calling
\[
B(U)=-\escpr{N,T}\escpr{J(U),\eta}+\escpr{U,T}\escpr{N_h,J(\eta)},
\]
we get
\begin{align*}
\tfrac{D}{ds}\big|_{s=0}||(N_s)_h||_*&=\big(\escpr{-\nabla_Su-A_S(U^\top),\eta}\big)_p+B(U_p)
\\
&=-\escpr{\nabla_Su,\eta}_p+B(U^\bot_p)+\big(-\escpr{A_S(U^\bot),\eta}_p+B(U^\top_p)\big)
\\
&=\big(-\escpr{\nabla_Su,\eta}-2u\escpr{N,T}\escpr{J(N_h),\eta}\big)_p+U^\top_p(||N_h||_*).
\end{align*}
%
%
Observe that
\begin{align*}
-\escpr{\nabla_Su,\eta}&=u\divv_S\eta-\divv_S(u\eta)
\\
&=u\divv_S\eta-\divv_S(u\eta^\top)-\divv_S(u\escpr{N,\eta}N)
\\
&=u\divv_S\eta-\divv_S(u\eta^\top)-u\norm{N_h}_*\divv_S N.
\end{align*}
Hence we get
\begin{equation*}
\begin{split}
A'(0)=\int_S\big(u\divv_S\eta&-2u\escpr{N,T}\escpr{J(N_h),\eta}\big)\,dS
\\
&+\int_S \divv_S\big(\norm{N_h}_*U^\top-u\eta^\top\big)\,dS.
\end{split}
\end{equation*}
From here we obtain formula \eqref{eq:A'(0)} since the integral $\int_S \norm{N_h}_*U^\top dS$ is equal to $0$ by the divergence theorem for Lipschitz vector fields.
\end{proof}

Now we simplify the first term appearing in the first variation formula \eqref{eq:A'(0)}.

\begin{lemma}
\label{lem:defH}
Let $S$ be a $C^2$ surface immersed in $\hh^1$ with unit normal $N$ horizontal unit normal $\nu_h$. Let $Z=J(\nu_h)$. Then we have
\begin{equation}
\label{eq:defH0}
\divv_S\eta-2\escpr{N,T}\escpr{J(N_h),\eta}=\escpr{D_Z\eta,Z}.
\end{equation}
\end{lemma}

\begin{proof}
Let us consider the orthonormal basis in $S\setminus S_0$ given by the vector fields $Z=-J(\nu_h)$ and $E=\langle N,T\rangle \nu_h -|N_h| T=a\nu_h+b T$. Using equation \eqref{eq:corconnection}  with $F=E$, we get
\begin{equation*}
\begin{split}
\escpr{D_E\eta,E}&=a\escpr{D_E\eta,\nu_h}+b\escpr{D_E\eta,T}
\\
&=-a\escpr{E,T}\escpr{\eta,J(\nu_h)}+b\big(E(\escpr{\eta, T})-\escpr{\eta,D_E T}\big)\\
&=-ab\escpr{\eta,J(\nu_h)}-ab\escpr{\eta,J(\nu_h)}\\
&=-2ab\escpr{\eta,J(\nu_h)},
\end{split}
\end{equation*}
as $D_E T=J(E)=a J(\nu_h)=-aZ$. From $ab=-\escpr{N,T}|N_h|$ we obtain
\[
\escpr{D_E\eta,E}=2\escpr{N,T}\escpr{\eta,J(N_h)}.
\]
Taking into account this equation and that $\divv_S\eta=\escpr{D_Z\eta,S}+\escpr{D_E\eta,E}$, we obtain equation \eqref{eq:defH0}.
\end{proof}

\begin{definition}
Given an oriented surface $S$ immersed in $\hh^1$ endowed with a smooth strictly convex left-invariant norm $\norm{\cdot}_K$, its mean curvature is the function 
\begin{equation}
\label{eq:defH}
H=\escpr{D_Z\eta_K,Z},
\end{equation}
defined on $S\setminus S_0$. 
\end{definition}

\begin{remark}
In \cite{snchez2017subfinsler,1711.01585}, the author obtained an expression of the mean curvature of a $C^2$ surface in terms of a parametrization when $\hh^1$ is endowed with the left-invariant norm $\|\cdot\|_\infty$, and defined a notion of distributional mean curvature for polygonal norms.
\end{remark}

\begin{corollary}
\label{end:corH}
Let $S$ be an oriented $C^2$ surface immersed in $\hh^1$. Let $U$ be a $C^2$ vector field with compact support on $S\setminus S_0$, normal component $u=\escpr{U,N}$ and associated one-parameter group of diffeomorphisms $\{\varphi_s\}_{s\in\rr}$. Then
\[
\frac{d}{ds}\bigg|_{s=0} A(\varphi_s(S))=\int_S u H\,dS,
\]
where $H$ is the mean curvature of $S$ defined in \eqref{eq:defH}. 
\end{corollary}

By equation \eqref{eq:defH}, a unit speed horizontal curve $\Ga$ contained in the regular part of a surface $S$ satisfy the equation
\begin{equation}
\label{eq:GaH}
\escpr{\frac{D}{ds}\pi(J(\dot{\Ga})),\dot{\Ga}}=H,
\end{equation}
where $D/ds$ is the covariant derivative along $\Ga$. Uniqueness of curves $\Ga$ satisfying \eqref{eq:GaH} with given initial conditions $\Ga(0),\dot{\Ga}(0)$ cannot be obtained from \eqref{eq:GaH}. In the next result we prove that the horizontal components of $\Ga$ satisfy indeed an ordinary differential equation, thus providing uniqueness with given initial conditions. 

\begin{corollary}\label{cor}
Let $S$ be a $C^2$ oriented surface immersed in $(\hh^1,\norm{\cdot})$ with mean curvature $H$. Let $\Gamma:I\to S\setminus S_0$ be a horizontal curve in the regular part of $S$ parameterized by arc-length with $\Gamma(s)=(x_1(s),x_2(s),t(s))$. Then $\ga(s)=(x_1,x_2)$ satisfies a differential equation of the form
\begin{equation}
\label{eq:F}
\ddot{\ga}=F(\dot{\ga}),
\end{equation}
where $F(\dot{\ga})=H\,[A(\dot{\ga})](\dot{\ga})$ and $A$ is a nonsingular $C^1$ matrix of order $2$.
\end{corollary}	

\begin{proof}
Let $\tfrac{D}{ds}$ be the covariant derivative along the curve $\Gamma$. Since $\Gamma$ is horizontal and parameterized by arc-length, the vector field $\tfrac{D}{ds}\dot{\Gamma}$ along $\Gamma$ is proportional to $J(\dot{\Gamma})$. Then there exists a function $\la:I\to\rr$ such that
\[
\tfrac{D}{ds}\dot{\Gamma}=\la J(\dot{\Gamma}).
\]
Taking scalar product with $\eta=\pi(J(\dot{\Gamma}))$ we get
\[
\la=\frac{\escpr{\tfrac{D}{ds}\dot{\Gamma},\pi(J(\dot{\Gamma}))}}{||J(\dot{\Gamma})||_*}=\frac{\tfrac{d}{ds}\escpr{\dot{\Gamma},\pi(J(\dot{\Gamma}))}-H}{||J(\dot{\Gamma})||_*}.
\]
Hence we have
\begin{equation}
\label{eq:ddga}
||J(\dot{\Gamma})||_* \tfrac{D}{ds}\dot{\Gamma}-\dot{f}\,J(\dot{\Gamma})=-HJ(\dot{\Gamma}),
\end{equation}
where $f=\escpr{\dot{\Gamma},\pi(J(\dot{\Gamma}))}$. Since $\dot{\Gamma}=\dot{x}_1X+\dot{x_2}Y$, $\tfrac{D}{ds}\dot{\Gamma}=\ddot{x}_1X+\ddot{x}_2 Y$, and $J(\dot{\Gamma})=-\dot{x}_2X+\dot{x}_1Y$, equation \eqref{eq:ddga} is equivalent to the system
\begin{equation}
\label{eq:ddga-system}
\begin{split}
||J(\dot{\Gamma})||_*\, \ddot{x}_1+\dot{f}\dot{x}_2&=H\dot{x}_2,
\\
||J(\dot{\Gamma})||_*\, \ddot{x}_2-\dot{f}\dot{x}_1&=-H\dot{x}_1.
\end{split}
\end{equation}
Let us compute $\dot{f}=df/ds$. Writing $\pi(aX+bY)=\pi_1(a,b)X+\pi_2(a,b)Y$ we have
\[
f=\escpr{\dot{\Gamma},\pi(J(\dot{\Gamma}))}=\dot{x}_1\pi_1(-\dot{x}_2,\dot{x}_1)+\dot{x}_2\pi_2(-\dot{x}_2,\dot{x}_1)
\]
and so:
\begin{equation*}
\dot{f}=\bigg(\pi_1+\dot{x}_1\frac{\ptl\pi_1}{\ptl x_2}+\dot{x}_2\frac{\ptl\pi_2}{\ptl x_2}\bigg)\,\ddot{x}_1+\bigg(\pi_2-\dot{x}_1\frac{\ptl\pi_1}{\ptl x_1}-\dot{x}_2\frac{\ptl\pi_2}{\ptl x_1}\bigg)\,\ddot{x}_2=g\ddot{x}_1+h\ddot{x}_2,
\end{equation*}
where the functions $\pi_1,\pi_2$ are evaluated at $(-\dot{x}_2,\dot{x}_1)$. Hence equation \eqref{eq:ddga-system} is equivalent to
\begin{equation}
\label{eq:ddga-system-2}
\begin{split}
\begin{pmatrix}
||J(\dot{\Gamma})||_*+g\dot{x}_1 & h \dot{x}_2 \\
-g\dot{x}_1 & ||J(\dot{\Gamma})||_*-h\dot{x}_1
\end{pmatrix}
\begin{pmatrix}
\ddot{x}_1 \\ \ddot{x}_2
\end{pmatrix}
=H
\begin{pmatrix}
\dot{x}_2 \\ -\dot{x}_1
\end{pmatrix}
\end{split}
\end{equation}
The determinant of the square matrix in \eqref{eq:ddga-system-2} is equal to
\[
||J(\dot{\Gamma})||_*\big(||J(\dot{\Gamma})||_*+(g\dot{x}_1-h\dot{x}_1)\big).
\]
Since
\begin{align*}
g\dot{x}_1-h\dot{x}_2&=\big(\pi_1\dot{x}_2-\pi_2\dot{x}_1\big)+\sum_{i,j=1}^2\dot{x}_i\dot{x}_j\frac{\ptl\pi_i}{\ptl x_j}=-||J(\dot{\Gamma})||_*+\sum_{i,j=1}^2\dot{x}_i\dot{x}_j\frac{\ptl\pi_i}{\ptl x_j}
\end{align*}
we get that the determinant is equal to
\[
||J(\dot{\Gamma})||_*\sum_{i,j=1}^2\dot{x}_i\dot{x}_j\frac{\ptl\pi_i}{\ptl x_j}
\]
and we write 
\[
\sum_{i,j=1}^2\dot{x}_i\dot{x}_j\frac{\ptl\pi_i}{\ptl x_j}=
\begin{pmatrix}
\dot{x}_1 & \dot{x}_2
\end{pmatrix}
\begin{pmatrix}
\ptl\pi_1/\ptl x_1 & \ptl\pi_1/\ptl x_2
\\
\ptl\pi_2/\ptl x_1 & \ptl\pi_2/\ptl x_2
\end{pmatrix}
\begin{pmatrix}
\dot{x}_1 \\ \dot{x}_2
\end{pmatrix}.
\]
Since the kernel of $\big(\ptl\pi_i/\ptl x_j\big)_{ij}$ is generated by $(-\dot{x}_2,\dot{x}_1)$, we have
\[
\begin{pmatrix}
\ptl\pi_1/\ptl x_1 & \ptl\pi_1/\ptl x_2
\\
\ptl\pi_2/\ptl x_1 & \ptl\pi_2/\ptl x_2
\end{pmatrix}
\begin{pmatrix}
\dot{x}_1 \\ \dot{x}_2
\end{pmatrix}\neq 0,
\]
and, since the image of $\big(\ptl\pi_i/\ptl x_j\big)_{ij}$ is generated by $(\dot{x}_1,\dot{x}_2)$, we get
\[
\begin{pmatrix}
\dot{x}_1 & \dot{x}_2
\end{pmatrix}
\begin{pmatrix}
\ptl\pi_1/\ptl x_1 & \ptl\pi_1/\ptl x_2
\\
\ptl\pi_2/\ptl x_1 & \ptl\pi_2/\ptl x_2
\end{pmatrix}
\begin{pmatrix}
\dot{x}_1 \\ \dot{x}_2
\end{pmatrix}\neq 0.
\]
So we can invert the matrix in  \eqref{eq:ddga-system-2} to get \eqref{eq:F}.
\end{proof}

\begin{remark}
It is not difficult to prove that
\[
\tfrac{D}{ds}\pi(J(\dot{\Gamma}))=H\dot{\Gamma}-||J(\dot{\Gamma})||_*\,T.
\]
Indeed it is only necessary to show that $\escpr{\tfrac{D}{ds}\pi(J(\dot{\Gamma})),J(\dot{\Gamma})}=0$, which follows from \eqref{eq:corconnection-2} using that $J(\dot{\Ga})=\nu_h$. Observe that the above equation is equivalent to
\[
\big[\tfrac{D}{ds}\pi(J(\dot{\Gamma}))\big]_h=H\dot{\Gamma}.
\]
Writing $\dot{\Ga}=\dot{x}X+\dot{y}Y$, we have
\begin{equation*}
\begin{pmatrix}
\ptl\pi_1/\ptl x_1 & \ptl\pi_1/\ptl x_2
\\
\ptl\pi_2/\ptl x_1 & \ptl\pi_2/\ptl x_2
\end{pmatrix}
\begin{pmatrix}
-\ddot{y} \\ \ddot{x}
\end{pmatrix}
=H
\begin{pmatrix}
\dot{x} \\ \dot{y}
\end{pmatrix}.
\end{equation*}
However, since the determinant of the square matrix is $0$ we cannot invert it to obtain an ordinary differential equation for $(\ddot{x},\ddot{y})$.
\end{remark}

\begin{lemma}
\label{eq:wulffcurve}
Let $\norm{\cdot}$ be a  $C^2_+$ left-invariant norm in $\hh^1$. Let $\ga:I\to\rr^2$ be a unit speed clockwise parameterization of a translation of the unit sphere of $\norm{\cdot}$ in $\rr^2$ by a vector $v\in\rr^2$. Let $\Gamma$ be a horizontal lifting of $z$. Then $\Gamma$ satisfies the equation
\begin{equation}
\label{eq:wulff-curves}
1=\escpr{\tfrac{D}{ds}\pi(J(\dot{\Gamma})),\dot{\Gamma}}.
\end{equation}
\end{lemma}

\begin{proof}
We have $\pi(J(\dot{\Gamma}))=\pi_1(J(\dot{\ga}))X+\pi_2(J(\dot{\ga}))Y$. Since $J(\dot{\ga})$ is the outer normal to the unit sphere at $\ga-v$ we have $\ga-v=\big(\pi_1(J(\dot{\Gamma})),\pi_2(J(\dot{\Gamma}))\big)$. Hence $\tfrac{D}{ds}\pi(J(\dot{\Gamma}))=\dot{x}X+\dot{y}Y$ and we get \eqref{eq:wulff-curves}.
\end{proof}

\begin{lemma}
Let $\norm{\cdot}$ be a  $C^2_+$ left-invariant norm in $\hh^1$ and $\Gamma$ a horizontal curve parameterized by arc-length satisfying the equation $\escpr{\tfrac{D}{ds}\pi(J(\dot{\Gamma})),\dot{\Gamma}}=H$, with $H\in\rr$. Then $\sg(s)=h_\la(\Gamma(s/\la))$ is parameterized by arc-length and  $\escpr{\tfrac{D}{ds}\pi(J(\dot{\sg})),\dot{\sg}}=H/\la$.
\end{lemma}

\begin{proof}
We have $\dot{\sg}(s)=\dot{\Gamma}(s/\la)$ and $J(\dot{\sg}(s))=J(\dot{\Gamma}(s/\la))$.
\end{proof}

\begin{remark}
Horizontal straight lines are solutions of
\[
\escpr{\tfrac{D}{ds}\pi(J(\dot{\Gamma})),\dot{\Gamma}}=0
\]
since $\dot{\Gamma}$ is the restriction of a left-invariant vector field in $\hh^1$ and so they are $J(\dot{\Gamma})$ and $\pi(J(\dot{\Gamma}))$.
\end{remark}

\begin{theorem}
\label{thm:horizontal}
Let $\norm{\cdot}$ be a  $C^2_+$ left-invariant norm in $\hh^1$. Let $\Gamma$ be a horizontal curve satisfying the equation
\begin{equation}
\label{eq:eqhorcurveH}
\escpr{\tfrac{D}{ds}\pi(J(\dot{\Gamma})),\dot{\Gamma}}=H,
\end{equation}
for some $H\ge 0$. Then $\Gamma$ is either a horizontal straight line if $H=0$ or the horizontal lifting  of a dilation and traslation of a unit speed clockwise parameterization of the circle $\norm{\cdot}=1$ in $\rr^2$ in case $H> 0$.
\end{theorem}

\begin{proof}
Horizontal straight lines and horizontal liftings of translations and dilations of the unit circle $\norm{\cdot}=1$ in $\rr^2$ satisfy equation \eqref{eq:eqhorcurveH}. Uniqueness follow since the projection to $t=0$ satisfy equation \eqref{eq:F} and, by using translations and dilations, we can obtain any prescribed initial condition.
\end{proof}

\begin{remark}
The result in Theorem~\ref{thm:horizontal} includes that constant mean curvature surfaces for the sub-Riemannian area in the Heisenberg group are foliated by geodesics. This result can be found, with slight variations, in \cite{MR2165405,MR2481053,MR2983199,MR3412382,MR3406514}. 
\end{remark}

To finish this section we prove the following result, that holds trivially for variations supported in the regular part of $S$.

\begin{proposition} 
\label{prop:varform}
Let $S$ be a compact $C^2$ oriented surface in $(\hh^1,\norm{\cdot})$ enclosing a region $E$. Assume that $S$ has constant mean curvature $H$ and a finite number of singular points. Then
\begin{enumerate}
\item $S$ is a critical point of the sub-Finsler area for any volume-preserving variation.
\item $S$ is a critical point of the functional $A-H\,|\cdot|$.
\end{enumerate}
\end{proposition}

\begin{proof}
It is only necessary to prove that if $U$ is a smooth vector field with compact support in $\hh^1$ and $\{\varphi_s\}_{s\in\rr}$ is its associated flow, then
\[
\frac{d}{ds}\bigg|_{s=0} A(\varphi_s(S))=\int_S Hu\,dS.
\]
From formula \eqref{eq:A'(0)} this is equivalent to proving that 
\begin{equation*}
\int_S \divv_S\big(u\eta^\top\big)\,dS=0.
\end{equation*}
To compute the integral $\int_S u\eta^\top dS$ we consider the finite number of singular points $p_1,\ldots,p_n$, and take small disjoint balls $B_i(p_i)$ centered at the points $p_i$. For $\eps>0$ small enough so that the balls $B_\eps(p_i)$ are contained in $B_i$ we have
\[
\int_{S\setminus \bigcup_{i=1}^n B_\eps(p_i)}\divv u \eta^\top \,dS=\sum_{i=1}^n \int_{\ptl B_\eps(p_i)} \escpr{\xi_i,u\eta^\top }\,d(\ptl B_\eps(p_i)),
\]
where $\xi_i$ is the unit inner normal to $\ptl B_\eps(p_i)$. Since $u\eta^\top $ is bounded and the lengths of $\ptl B_\eps(p_i)$ go to $0$ when $\eps\to 0$ we have
\[
\lim_{\eps\to 0}\sum_{i=1}^n \int_{\ptl B_\eps(p_i)} \escpr{\xi_i,u\eta^\top }\,d(\ptl B_\eps(p_i))=0.
\]
Since the modulus of
\begin{align*}
\divv_S(u\eta^\top)&=\escpr{\nabla_S u,\eta ^\top}+u\divv_S \eta^\top
\\
&=\escpr{\nabla_S u,\eta^\top}+u\,(\divv_S \eta-\escpr{\eta^\top,N}\divv_S N)
\end{align*}
is uniformly bounded, the dominated convergence theorem implies
\begin{align*}
\int_S\divv_S u\eta^\top dS&=\lim_{\eps\to 0} \int_{S\setminus \bigcup_{i=1}^n B_\eps(p_i)}\divv u \eta^\top \,dS
\\
&=\lim_{\eps\to 0}\sum_{i=1}^n \int_{\ptl B_\eps(p_i)} \escpr{\xi_i,u\eta^\top }\,d(\ptl B_\eps(p_i))=0. \qedhere
\end{align*}
\end{proof}
%

\begin{corollary}[Minkowski formula]
\label{cor:minkowski}
Let $S$ be a compact $C^2$ oriented surface in $(\hh^1,\norm{\cdot})$ enclosing a region $E$. Assume that $S$ has constant mean curvature $H$ and a finite number of singular points. Then
\begin{equation}
\label{eq:minkowski}
3A(S)-4H|E|=0.
\end{equation}
\end{corollary}

\begin{proof}
We consider the vector field $W=x\tfrac{\ptl}{\ptl x}+y\tfrac{\ptl}{\ptl y}+2\tfrac{\ptl}{\ptl t}$ and its associated flow $\varphi_s((x,y,t))=(e^sx,e^sy,e^{2s}t)$. Since
\[
\frac{d}{ds}\bigg|_{s=0} A(\varphi_s(S))=3A(S),\quad  \frac{d}{ds}\bigg|_{s=0} |\varphi_s(E)|=4|E|,
\]
Proposition~\ref{prop:varform} implies
\[
0=\frac{d}{ds}\bigg|_{s=0} A(\varphi_s(S))-H\frac{d}{ds}\bigg|_{s=0} |\varphi_s(E)|=3A(S)-4H|E|. \qedhere
\]
\end{proof}

\section{Pansu-Wulff spheres and examples}
\label{sec:examples}

We consider a convex body $K\subset\rr^2$ containing $0$ in its interior and the associated norm
$\norm{\cdot}_K$ in $\hh^1$.

\begin{definition}
Consider a clockwise-oriented $L$-periodic parameterization $\ga:\rr\to\rr^2$ of the curve $\norm{\cdot}_K=1$. For fixed $v\in \rr$ take the translated curve $u\mapsto\ga(u+v)-\ga(v)$ and its horizontal lifting $\Ga_{v}:\rr\to\hh^1$ with initial point $(0,0,0)$ at $u=0$.

The set $\sph_K$ is defined as 
\begin{equation}
\sph_K=\bigcup_{v\in [0,L)} \Ga_{v}([0,L]).
\end{equation}
We shall refer to $\sph_K$ as the \emph{Pansu-Wulff sphere} associated to the left-invariant norm $\norm{\cdot}_K$.
\end{definition}

When $K=D$, the closed unit disk centered at the origin in $\rr^2$, the Pansu-Wulff sphere $\sph_D$ is Pansu's sphere, see \cite{MR676380,MR829003}.

\begin{remark}
In the construction of the Pansu-Wulff sphere we are not assuming any regularity on the boundary of $K$. Since $\ptl K$ is a locally Lipschitz curve, its horizontal lifting is well defined.
\end{remark}

\begin{remark}
The set $\sph_K$ is union of curves leaving from $(0,0,0)$ that meet again at the point $(0,0,2|K|)$. Since $\ga$ is $L$-periodic, the construction is $L$-periodic in $v$ and so $\sph_K$ is the image of a continuous map from a sphere to $\hh^1$.
\end{remark}

\begin{example}
Given the Euclidean norm $|\cdot|$ in $ \rr^2$ and $a=(a_1,a_2)$, where $a_1,a_2>0$, we define  the norm:
\[
||(x_1,x_2)||_a=|(\tfrac{x_1}{a_1},\tfrac{x_2}{a_2})|.
\]

\begin{figure}[h]
\includegraphics[width=0.5\textwidth]{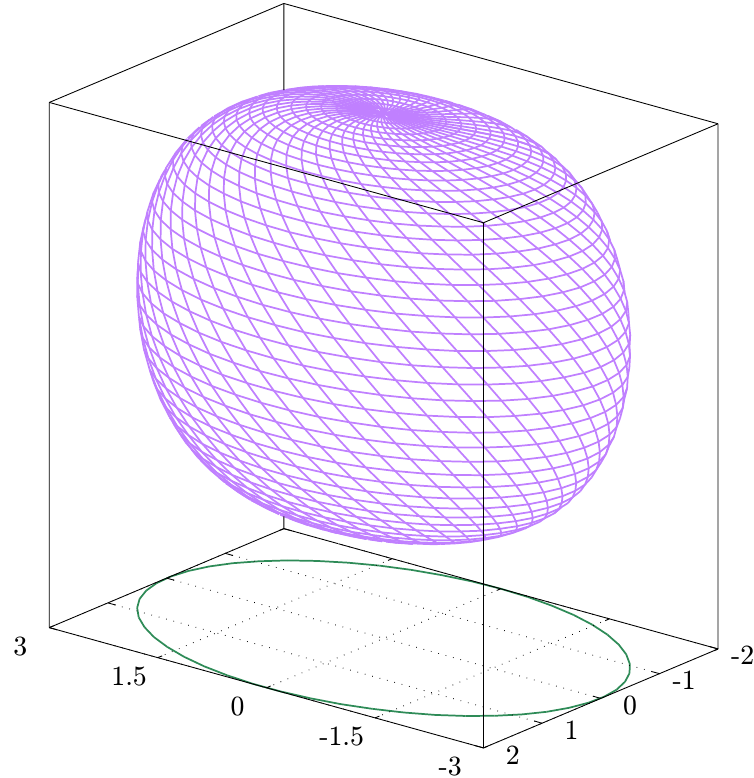}
\caption{The Pansu-Wulff sphere associated to the norm $\norm{\cdot}_a$ with $a=(1,1.5)$. Observe that the projection to the horizontal plane $t=0$ is an ellipse with semiaxes of lengths $2$ and $3$.}
\end{figure}

The unit ball $K_a$ for this norm is an ellipsoid with axes of length $a_1$ and $a_2$. We parameterize clockwise the unit circle of the norm $\norm{\cdot}_K$ by
\[
\ga(s)=(a_1\sin(s),a_2\cos(s)), \quad s\in\rr.
\]
This parameterization is injective of period $2\pi$. The translation of this curve to the origin by the point $-\ga(v)$ is given by the curve
\[
\Lambda_v(u)=\ga(u+v)-\ga(v).
\]
The horizontal lifting of $\Lambda_v$ is given by $(\Lambda_v(u),t_v(u))$, where
\[
t_v(u)=\int_0^u \big[\Lambda_v(\xi)\cdot J(\dot{\Lambda}_{v}(\xi))\big]\,d\xi.
\]
Since
\[
\Lambda_v(\xi)\cdot J(\dot{\Lambda}_{v}(\xi))=(\ga(\xi+v)-\ga(v))\cdot J(\dot{\ga}(\xi+v)),
\]
we get
\[
t_v(u)=a_1a_2\big(u+\sin(v)\cos(u+v)-\cos(v)\sin(u+v)\big).
\]

Hence a parameterization of $\sph_{K_a}$ is given by
\begin{align*}
x(u,v)&=a_1\big(\sin(u+v)-\sin(v)\big)
\\
y(u,v)&=a_2\big(\cos(u+v)-\sin(v)\big),
\\
t(u,v)&=a_1a_2\big(u+\sin(v)\cos(u+v)-\cos(v)\sin(u+v)\big).
\end{align*}
\end{example}

\begin{example}
Given any convex set $K$ containing $0$ in its interior, we can parameterize its Lipschitz boundary $\ptl K$ as
\[
\ga(s)=\big(x(s),y(s)\big)=r(s)\,\big(\sin(s),\cos(s)\big), \quad s\in\rr.
\]
where $r(s)=\rho(\sin(s),\cos(s))$ and $\rho$ is the radial function of $K$ defined as $\rho(u)=\sup\{\la\ge 0: \la u\in K\}$ for any vector $u$ of modulus $1$ in $\rr^2$.

A horizontal lifting of the curve $\ga$ passing through the point $(\ga(0),0)$ can be obtained computing 
\begin{equation*}
t(s)=\int_0^s \ga(\xi)\cdot J(\dot{\ga}(\xi))\,d\xi
=\int_0^s r^2(\xi)\,d\xi,
\end{equation*}
since $J(\dot{\ga}(s))=r(s)\,(\sin(s),\cos(s))+\dot{r}(s)\,(-\cos(s),\sin(s))$. Hence the curve
\[
\Ga(s)=\big(x(s),y(s),t(s)\big)=\big(\ga(s),\int_0^s r^2(\xi)\,d\xi\big)
\]
is a horizontal lifting of the curve $\ga$.

Now we translate all these curves to pass through the origin of $\hh^1$. This way we get the parameterization $\Phi_K$ of $\sph_K$ given by
\[
(u,v)\mapsto \ell_{-\Ga(v)} (\Ga(u+v))
\]
for $(u,v)\in [0,2\pi]^2$. Since
\[
\ell_{(x_0,y_0,t_0)}(x,y,t)=\big(x+x_0,y+y_0,t+t_0+(xy_0-x_0y)\big),
\]
computing the left-translation using the expression for $\Ga$ obtained before we get
\begin{equation}
\label{eq:parametricSK}
\begin{split}
x(u,v)&=r(u+v)\sin(u+v)-r(v)\sin(v),
\\
y(u,v)&=r(u+v)\cos(u+v)-r(v)\cos(v),
\\
t(u,v)&=r(v)r(u+v)\big(\sin(v)\cos(u+v)-\cos(v)\sin(u+v)\big)
\\
&\qquad +\textstyle\int_v^{u+v} r^2(\xi)\,d\xi.
\end{split}
\end{equation}

The parameterization given by equations \eqref{eq:parametricSK} is useful to obtain regularity properties of $\sph_K$. If $\ptl K$ is of class $C^\ell$, $\ell\ge 0$, its radial function $r(s)=(x(s)^2+y(s)^2)^{1/2}$ is of class $C^\ell$ and hence the parameterization $\Phi_K$ is an immersion of class $C^\ell$ for $0<u<2\pi$.
\end{example}

\begin{example}
Let $\ell>1$. We consider the $\ell$-norm in $\rr^2$ defined as
\[
\norm{(x_1,x_2)}_\ell=\big(|x_1|^\ell+|x_2|^\ell\big)^{1/\ell}.
\] 
Denote by $K_\ell$ the unit ball for this $\ell$-norm. We can parametrize the unit circle $\norm{\cdot}_\ell=1$ using \eqref{eq:parametricSK}. In this case
\[
\rho(x,y)=\frac{1}{\big(|x|^\ell+|y|^\ell\big)^{1/\ell}}, \qquad |(x,y)|=1.
\]
By the previous example, the Pansu-Wulff sphere $\sph_{K_\ell}$ is parameterized by equations \eqref{eq:parametricSK}.
\end{example}

\begin{remark}
\label{eq:uniformconvergence}
Assume we have a sequence of of convex sets $(K_i)$ converging in Hausdorff distance to a limit convex set $K$. Then the radial functions $r_{K_i}$ uniformly converge to the radial function $r$ of the limit set $K$. Hence equations \eqref{eq:parametricSK} imply that the Pansu-Wulff spheres $\sph_{K_i}$ converge in Hausdorff distance to a ball bounded by the horizontal liftings of translations of a parameterization $\ga$ of $\ptl K$.

Since $\lim_{\ell\to 1}\norm{\cdot}_\ell=\norm{\cdot}_1$ and $\lim_{\ell\to\infty}\norm{\cdot}_\ell=\norm{\cdot}_\infty$, we can use the previous argument to show that the Pansu-Wulff spheres $\sph_{K_\ell}$ converge to the two spheres $\sph_1$ and $\sph_\infty$. Under these conditions, it is not difficult to check that the corresponding perimeters converge to the limit perimeter.
\end{remark}

\begin{figure}[h]
\includegraphics{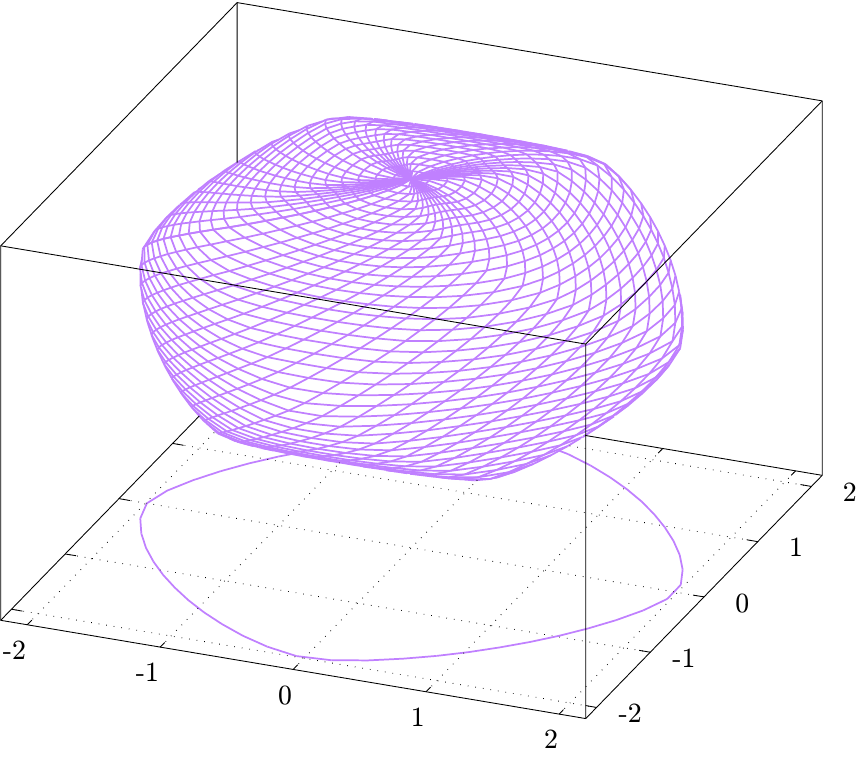}
\caption{The Pansu-Wulff sphere $\sph_{K_\ell}$ for the $\ell$-norm, $\ell=1.5$. The horizontal curve is the projection of the equator to the plane $t=0$. We observe that the Pansu-Wulff sphere projects to the set $\norm{\cdot}_\ell\le 2$ in the $t=0$ plane.}
\end{figure}

\begin{figure}[h]
\includegraphics{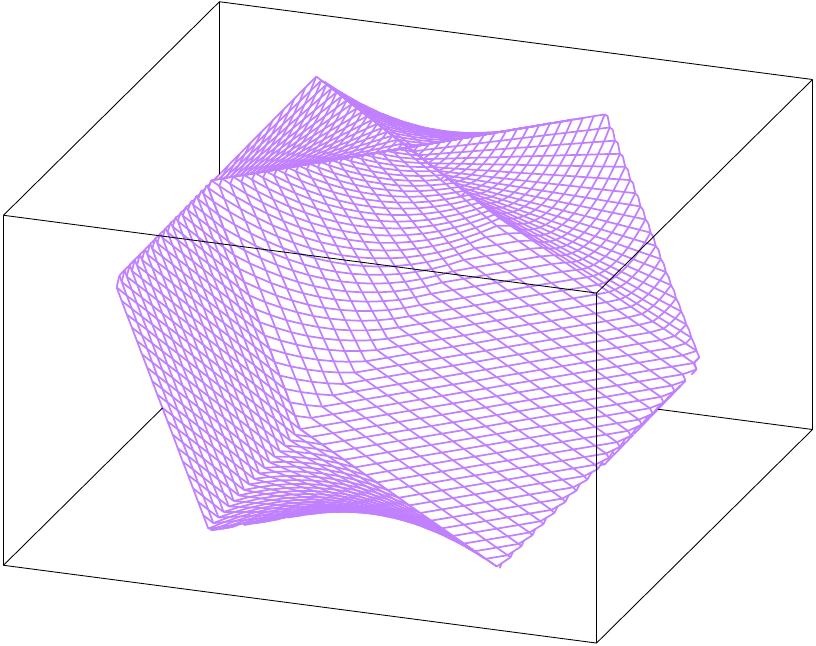}
\caption{The sphere $\sph_1$ obtained as Hausdorff limit of the Pansu-Wulff spheres $\sph_{K_r}$ of the $\ell$-norm when $\ell$ converges to $1$}
\end{figure}

\begin{figure}[h]
\includegraphics{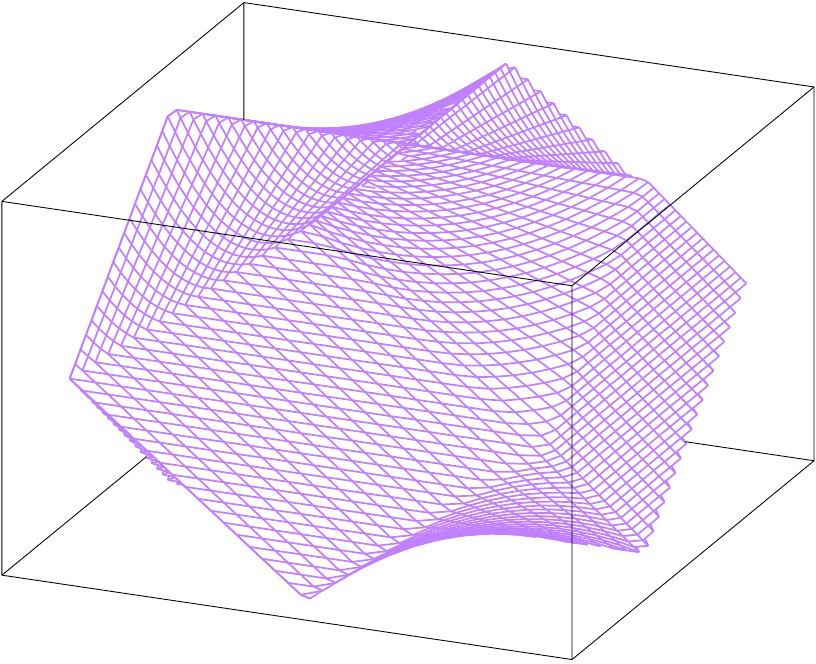}
\caption{The sphere $\sph_\infty$ obtained as Hausdorff limit of the Pansu-Wulff spheres $\sph_{K_r}$ of the $\ell$-norm when $\ell$ converges to $\infty$}
\end{figure}

\begin{example}
Let us consider the equilateral triangle $T$ in the plane $\rr^2$ defined as the convex envelope of the points $a_1=(0,1)$, $a_2=(\sqrt{3}/2,-1/2)$, $a_3=(-\sqrt{3}/1,-1/2)$. We can define a norm $\norm{\cdot}_T$ by the equality
\[
\norm{x}_T=\max\big\{\escpr{x,a_i}:i=1,2,3\big\},\quad x\in\rr^2.
\]
\begin{figure}[h]
\includegraphics{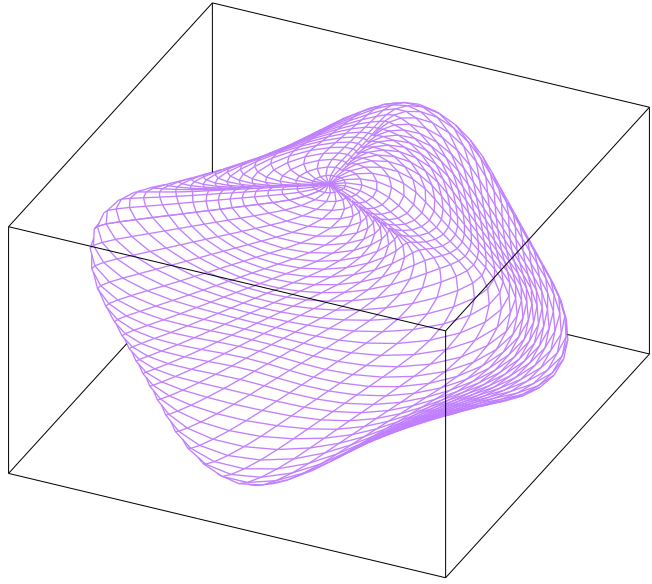}
\caption{The Pansu-Wulff sphere $\sph_{T,\ell}$ for the norm $\norm{\cdot}_{T,\ell}$, with $r=2$.}
\end {figure}
\begin{figure}[h]
\includegraphics{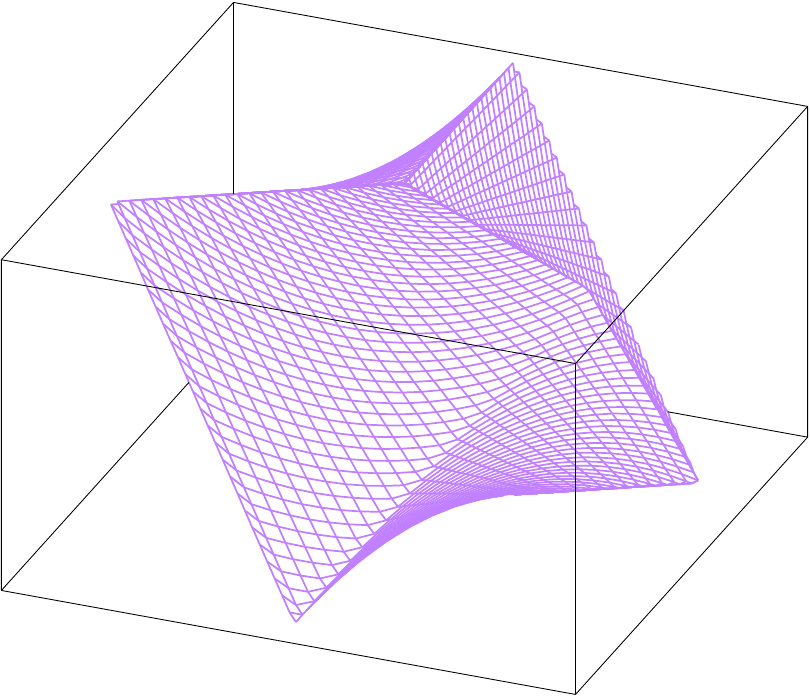}
\caption{The sphere $\sph_T$ obtained as limit of the Pansu-Wulff spheres $\sph_{T,\ell}$ when $r\to\infty$.}
\end {figure}

The unit ball of the norm $\norm{\cdot}_T$ is the triangle $T$. It is neither smooth nor strictly convex. However we may consider the approximating norms
\[
\norm{x}_{T,\ell}=\bigg(\sum_{i=1}^3 \max\{\escpr{x,a_i},0\}^\ell\bigg)^{1/\ell}.
\]
These norms are smooth and strictly convex and $\lim_{\ell\to\infty}\norm{x}_{T,\ell}=\norm{x}_T$. Hence the Pansu-Wulff spheres $\sph_{K_{T,\ell}}$ converge in Hausdorff distance when $\ell\to\infty$ to the sphere $\sph_T$ obtained by traslating $\ptl T$ to touch the origin and lifting the obtained curves as horizontal ones to $\hh^1$.
\end{example}
%

\section{Geometric properties of the Pansu-Wulff spheres}
\label{sec:geomprop}

In this section we show several geometric properties of the Pansu-Wulff spheres $\sph_K$ associated with a left-invariant norm $\norm{\cdot}_K$ . We start by looking at the projection of the sphere to the $t=0$ plane. This projection is determined by the geometry of the convex set $K$. 

Given a convex body $K\subset\rr^n$, the \emph{difference body} of $K$ is the set
\[
\text{D}K=K-K=\{x-y: x,y\in K\}.
\]
The difference body D$K$ is a \emph{centrally symmetric} convex body. This means that $-x\in\text{D}K$ whenever $x\in\text{D}K$. If $h_K$ is the support function of $K$ then the support function of $\text{D}K$ is given by 
\[
h_{\text{D}K}(u)= h_K(u)+h_K(-u), 
\]
see \cite[p.~140]{MR3155183}. This is the width of $K$ in the direction of $u$.

\begin{lemma}
\label{lem:proj}
Let $K\subset\rr^n$ be a convex body with $0\in\intt(K)$. We consider the set
\begin{equation}
\label{eq:K0}
K_0=\bigcup_{p\in\ptl K} (-p+K).
\end{equation}
Then we have
\begin{enumerate}
\item $0\in K_0$.
\item $K_0$ is a convex body.
\item $K_0$ is the difference body of $K$. In particular, $K_0$ is centrally symmetric.
\item If $K$ is centrally symmetric then $K_0=2K$.
\item We have
\[
\bigcup_{p\in\ptl K} (-p+K)=\bigcup_{p\in\ptl K} (-p+\ptl K).
\]
\end{enumerate}
\end{lemma}

\begin{proof}
To prove 1 take into account that $0=-p+p\in -p+K\subset K_0$ for any $p\in\ptl K$.

To prove 2, we take $p_1,p_2\in\ptl K$, $q_1,q_2\in K$ and $\la\in [0,1]$. Then
\[
\la(-p_1+q_1)+(1-\la)(-p_2+q_2)=-p_\la+q_\la,
\]
where
\[
p_\la=\la p_1+(1-\la)p_2,\qquad q_\la=\la q_1+(1-\la)q_2.
\]

If $p_\la=q_\la$ then $-p_\la+q_\la=0\in K_0$. Otherwise the segment $[p_\la,q_\la]$ is not trivial and contained in $K$. Let $\mu_0\ge 1$ such that $q_\la+\mu_0(p_\la-q_\la)\in\ptl K$. The value $\mu_0$ is computed as the supremum of the set $\{\mu\ge 0: q_\la+\mu(p_\la-q_\la)\in K\}$. We have
\[
-p_\la+q_\la=-(q_\la+\mu_0(p_\la-q_\la))+(q_\la+(\mu_0-1)(p_\la-q_\la)).
\]
The point $q_\la+\mu_0(p_\la-q_\la)$ belongs to $\ptl K$ by the choice of $\mu_0$ and the point $q_\la+(\mu_0-1)(p_\la-q_\la)$ belongs to $K$ since $0\le \mu_0-1\le \mu_0$. Hence $-p_\la+q_\la\in K_0$ and so $K_0$ is convex.

To prove 3, we take a vector $v$ with $\escpr{v,v}=1$. Let $q\in\ptl K_0$ such that
\begin{equation}
\label{eq:defqh}
h_{K_0}(v)=\escpr{q,v}\ge \escpr{z,v}\quad \forall\ z\in K_0.
\end{equation}
By the definition of $K_0$, there exists $p\in\ptl K$ such that $q\in -p+K$. We claim that $q\in -p+\ptl K$: otherwise $p+q\in\intt(K)$ and there exists $\eps>0$ such that $p+q+\eps v\in K$. So we have
\[
\escpr{-p+(p+q+\eps v),v}=\escpr{q+\eps v,v}=\escpr{q,v}+\eps>\escpr{q,v}.
\]
Since $p+q+\eps v\in K$ this yields a contradiction. Hence $q\in -p+\ptl K=\ptl (-p+K)$ for some $p\in\ptl K$.

Since $-p+K\subset K_0$ and $q$ is a boundary point for both sets, we deduce that $v$ is a normal vector to $-p+K$ at $q$. As $h_{-p+K}(v)=-\escpr{p,v}+h_K(v)$,we have
\[
h_{K_0}(v)=h_{-p+K}(v)=h_K(v)+\escpr{p,-v}.
\]

It remains to prove that $h_{K}(-v)=\escpr{p,-v}$. Assume by contradiction that $\escpr{p,-v}<h_K(-v)=\escpr{x,-v}$ for some $x\in\ptl K$. Then we have
\[
\escpr{-x+(p+q),v}=\escpr{-x+p,v}+\escpr{q,v}>\escpr{q,v},
\]
that cannot hold by \eqref{eq:defqh} since $p+q\in K$ and so $-x+p+q\in -x+K\subset K_0$.

To prove 4, we note that $h_K(v)=h_K(-v)$ when $K$ is centrally symmmetric and, by 3, $h_{K_0}=2h_K$. Hence $K=2K_0$.

Finally, to prove 5 we notice that $\bigcup_{p\in\ptl K} (-p+K)\supset \bigcup_{p\in\ptl K} (-p+\ptl K)$. To prove the remaining inclusion we take $p\in\ptl K$ and $u\in K$ such that $q=-p+u\in \bigcup_{p\in\ptl K} (-p+K)$. Then Lemma~\ref{lem:segmentconvex} allows us to find $p_1,u_1\in \ptl K$ such that $q=-p+u=-p_1+u_1$. Hence $q\in \bigcup_{p\in\ptl K}  (-p+\ptl K)$.
\end{proof}

\begin{lemma}
\label{lem:segmentconvex}
Let $K\subset\rr$ be a convex body, and $a,b\in K$. Then there exist $p,q\in\ptl K$ such that $b-a=q-p$. 
\end{lemma}

\begin{proof}
	If $a=b$ or $a,b\in\ptl K$ the result follows trivially. Henceforth we assume $a\neq b$ and that at least $a$ or $b$ is an interior point of $K$. We pick a point $c\in K$ out of the line $ab$. Let $P$ be the plane containing $a,b,c$ and $W=K\cap P$. The set $W$ is a convex body in $P$ and the boundary of $W$ in $P$ is contained in $\ptl K$. We take orthogonal coordinates $(x,y)$ in $P$ so that $(b-a)$ points into the positive direction of the $y$-axis. Let $I$ be the orthogonal projection in $P$ of $W$ onto the $x$-axis.
	
	Given $x\in I$, define the set $W(x)$ as $\{y\in\rr: (x,y)\in W\}$. A simple application~of Kuratowski criterion, see Theorem~1.8.8 in \cite{MR3155183}, implies that $W(x_i)$ converges to $W(x)$ in Hausdorff distance when $x_i$ converges to $x$. Hence the function $x\in I\mapsto |W(x)|$ is continuous and takes a value larger than $||b-a||$ at the projection of $a,b$ over the $x$-axis. If $|W(x)|=||b-a||$ for some $x\in I$, we take as $p,q$ the extreme points of the interval  $W(x)$ chosen so that $q-p=b-a$ to conclude the proof. Otherwise, we would have $|W(x_0)|> ||b-a||$ at an extreme point $x_0$ of $I$.  We may choose two points $p,q\in W(x_0)$ such that $|[p,q]|=||b-a||$ and $q-p=b-a$. Since $W(x_0)$ is contained in the boundary of $W$ in $P$, it is contained in $\ptl K$ and so $p,q\in\ptl K$.
\end{proof}

Now we refine the results in Lemma~\ref{lem:proj} when $K$ is strictly convex and has boundary of class $C^\ell_+$, $\ell\ge 2$. We say that a convex body $K$ is of class $C^\ell_+$, $\ell\ge 1$, when $\ptl K$ is of class $C^\ell$ and its normal map $N_K:\ptl K\to\sph^1$ is a diffeomorphism of class $C^{\ell-1}$.

\begin{corollary}
Let $K\subset\rr^2$ be a convex body containing $0$ as interior point. Then
\begin{enumerate}
\item If $K\subset\rr^2$ is strictly convex, then $K_0$ is strictly convex.
\item If $K$ is of class $C^\ell_+$, $\ell\ge 2$, then $K_0$ is of class $C^\ell_+$.
\end{enumerate}
\end{corollary}

\begin{proof}
To prove that $K_0$ is strictly convex, we take two different points $x_1-x_2, y_1-y_2\in\ptl K_0$, with $x_i,y_i\in K$, $i=1,2$. Then the four points belong to the boundary of $K$. For any $\la\in (0,1)$, we write the convex combination $\la(x_1-x_2)+(1-\la)(y_1-y_2)$ as
\[
x_\la-y_\la=(\la x_1+(1-\la)y_1)-(\la x_2+(1-\la)y_2).
\]
Since $x_1\neq y_1$ or $x_2\neq y_2$, the strict convexity of $K$ implies that $x_\la$ or $y_\la$ is an interior point of $K$. Then $x_\la-y_\la$ is an interior point of $K_0$. Since $\la\in (0,1)$ and the boundary points are arbitrary, the set $K_0$ is strictly convex.

To prove the boundary regularity of $K_0$ we follow Schneider's arguments \cite[p.~115]{MR3155183} and observe that the support function $h_K$ of $K$ is defined, when $u\neq 0$, by
\[
h_K(u)=\escpr{u,N_K^{-1}(u)},
\]
where $N_K:\ptl K\to\sph^1$ is the Gauss map, a diffeomorphism of class $C^{\ell-1}$ since $K$ is of class $C^\ell_+$. By Corollary~7.1.3 in \cite{MR3155183} 
\begin{equation}
\label{eq:gradhK}
\nabla h_K(u)=N_K^{-1}\bigg(\frac{u}{|u|}\bigg),
\end{equation}
and so $h_K$ is of class $C^\ell$. This implies that the support function of $K_0$, $h_{K_0}(u)=h_K(u)+h_K(-u)$, is of class $C^\ell$. Hence the polar body $K_0^*$ of $K_0$ has boundary of class $C^\ell$.  The Gauss map $N_{K_0^*}$ of $K_0^*$ can be described as
\[
N_{K_0^*}:\rho(K_0^*,u)u\mapsto \frac{N_K^{-1}(u)}{|N_K^{-1}(u)|},
\]
where $\rho(K_0^*,\cdot)=h_K^{-1}(\cdot)$ is the radial function of $K_0^*$, of class $C^{\ell-1}$. Hence $N_{K_0^*}$ is a diffeomorphism of class $C^{\ell-1}$ and so $K_0^*$ is of class $C^\ell_+$. Now the support function of $K_0^*$ is of class $C^\ell_+$ and we reason in the same way interchanging the roles of $K_0^*$ and $K_0$ to get the result.
\end{proof} 

\begin{remark}
If $K\subset\rr^2$ is a centrally symmetric convex body, for any $p\in\ptl K$, the line passing through $p$ and $-p$ divides $K$ into two regions of equal area. Hence the line through $0$ and $-2p$ divides $-p+K$ into two regions of the same area. When $p$ moves along $\ptl K$, the point $-2p$ parametrizes $\ptl (2K)$.
\end{remark}

Let $K$ be a convex set of class $C^\ell_+$, $\ell\ge 2$, $C=\ptl K$ and $\ga:\rr\to \rr^2$ an $L$-periodic  clockwise arc-length parameterization of $C$, with $L=\text{length}(C)$. The set $K_0=\bigcup_{p\in C} (-p+K)$ has smooth boundary $C_0$. For any $v\in\rr$, we denote by $\ga_v(u)=\ga(u+v)-\ga(v)$. Let $\Ga_v=(\ga_v,t_v)$ be the horizontal lifting of $\ga_v$ with $t_v(0)=0$. If we call $\Om_v(u)$ the planar region delimited by the segment $[0,\ga_v(u)]$ and the restriction of $\ga_v$ to $[0,u]$ then a standard application of the Divergence Theorem to the vector field $x\tfrac{\ptl}{\ptl x}+y\tfrac{\ptl}{\ptl y}$ implies
\[
t_v(u)=\int_0^u \escpr{\ga_v,J(\dot{\ga}_v)}(\xi)\,d\xi=2\,|\Om_v(u)|.
\]

Our next goal is to prove that $\sph_K$ is the union of two graphs defined in $K_0$ of class $C^2$ and coinciding on $\ptl K_0$.

\begin{theorem}\label{thm:graph}
Let $K\subset\rr^2$ be a convex body with $C^\ell_+$ boundary, $\ell\ge 2$. Then
\begin{enumerate}
\item $\sph_K$ is of class $C^\ell$ outside the poles.
\item There exist two functions $g_1,g_2:K_0\to\rr$ of class $C^\ell$ on $\intt(K_0)$ such that
\[
\mathbb{S}_K=\text{\emph{graph}}(g_1)\cup \text{\emph{graph}}(g_2),
\]
with $g_1> g_2$ on $\intt(K_0)$ and  $g_1=g_2$ on $C_0$. This imples that $\sph_K$ is an embedded surface.

Moreover, if $K$ is centrally symmetric then $g_1+g_2=2|K|$ and hence $\sph_K$ is symmetric with respect to the horizontal Euclidean plane $t=|K|$.
\end{enumerate}
\end{theorem}

\begin{definition}
The domain delimited by the embedded sphere $\sph_K$ is a ball $\bb_K$ that we call the \emph{Pansu-Wulff shape} of $\norm{\cdot}_K$.
\end{definition}

\begin{proof}[Proof of Theorem~\ref{thm:graph}]
That $\sph_K$ is $C^\ell$ outside the singular set follows from the parameterization \eqref{eq:parametricSK} since the function $r(s)$ is of class $C^\ell$. This proves 1.

We break the proof of 2 into several steps. Recall that $C=\ptl K$ and $C_0=\ptl K_0$.

\emph{Step 1}. Given $x\in K_0\backslash\{0\}$, we claim that $x\in C-p$ for some $p\in C$ if and only if the segment $[p,p+x]$ is contained in $K$ and $p,p+x\in C$. This means that the number of curves $C-p$, with $p\in C$, passing through $x\neq 0$ coincides with the number of segments parallel to $x$ of length $|x|$ and boundary points in $C$. This step is trivial.

\emph{Step 2}. Given $x\in K_0\setminus\{0\}$, the number of segments $[p,p+x]$ contained in $K$ with $p,p+x\in C$ is either $1$ or $2$. The first case corresponds to maximal length and happens if and only if $x$ belongs to $C_0$.

To prove this we consider $v=x/|x|$ and a line $L$ orthogonal to $v$. For any $z$ in $L$ we consider the intersection $I_z=L_z\cap K$, where $L_z$ is the line passing through $z$ with direction $v$. The set $J=\{z\in L: I_z\neq\emptyset\}$ is a non-trivial segment in $L$. The strict convexity of $K$ implies that the map $F:J\to\rr$ defined by $F(z)=|I_z|$ is strictly concave.  Since $F$ vanishes at the extreme points of $J$, it has just one maximum point $z_0\in\intt(J)$ and each value in the interval $(0,F(z_0))$ is taken by two different points in $J$. The observation that there is a bijective correspondence between the segments $[p,p+x]$ contained in $K$ with $p,p+x\in C$ and the points $z\in L$ with $F(z)=|x|$ proves the first part of the claim.
\begin{figure}[h]
\begin{tikzpicture}
	\begin{pgfonlayer}{nodelayer}
		\node [style=none] (0) at (0.5, 0.525) {};
		\node [style=none] (1) at (-1.75, 0.025) {};
		\node [style=none] (2) at (-0.5, -1.5) {};
		\node [style=none] (3) at (1.75, -1) {};
		\node [style=none] (4) at (-3.75, -0.5) {};
		\node [style=none] (5) at (0.5, -3.25) {};
		\node [style=none] (9) at (-1, 0.825) {$K$};
		\node [style=none] (10) at (-3.25, -1.225) {$L$};
		\node [style=none] (11) at (1.5, 1) {$L_z$};
		\node [style=none] (12) at (-0.325, -3.05) {$z$};
		\node [style=none] (13) at (0.275, -2.125) {$p$};
		\node [style=none] (14) at (1.975, 0.1) {$p+x$};
		\node [style=none] (15) at (2, 1.025) {};
		\node [style=none] (16) at (-0.75, -3.25) {};
		\node [style=black dot] (17) at (1.35, 0) {};
		\node [style=black dot] (18) at (0.225, -1.725) {};
		\node [style=black dot] (20) at (-0.375, -2.675) {};
	\end{pgfonlayer}
	\begin{pgfonlayer}{edgelayer}
		\draw [in=165, out=90, looseness=0.75] (1.center) to (0.center);
		\draw [in=60, out=-15, looseness=0.75] (0.center) to (3.center);
		\draw [in=-30, out=-120] (3.center) to (2.center);
		\draw [in=-90, out=150, looseness=0.75] (2.center) to (1.center);
		\draw (4.center) to (5.center);
		\draw (15.center) to (16.center);
	\end{pgfonlayer}
\end{tikzpicture}
\caption{Construction of the map $F$}
\end{figure}
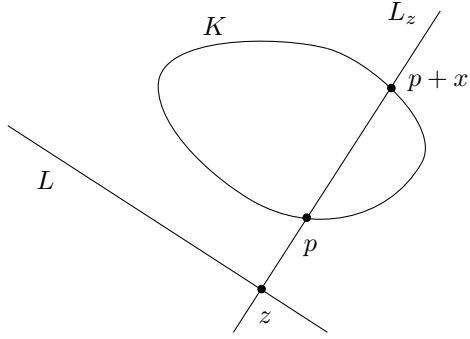

To prove the second part of the claim we fix some $x\in K_0$. We take $p\in C$ such that the segment $[p,p+x]$ is contained in $K$ and $p,p+x\in C$. Assume first that $x\in C_0$. If there were a larger segment $[q,q+\mu x]$ contained in $K$ with $q,q+\mu x\in C$ and $\mu>1$ then we would have $\mu x\in C-q\subset K_0$, a contradiction. Hence the length of $[p,p+x]$ is is the largest possible in the direction of $x$. Assume now that the length of $[p,p+x]$ yields the maximum of length of intervals contained in $K$ in the direction of $x$. If $x\not\in C_0$ then $x$ is an interior point of $K_0$ and, since $0\in\intt(K_0)$, there would exist $\la>1$ such that $\la x\in K_0$. Hence there is some $q\in C$ such that $\la x\in C-p$ and the segment $[q,q+\la x]\subset C$ and has length larger than $|x|$, a contradiction that proves that $x\in C_0$.

\emph{Step 3}. Given any point $x\in\intt(K_0)$, there are exactly two points in $\sph_K$ at heights $g_1(x)>g_2(x)$. In case $K$ is centrally symmetric then $g_1(x)+g_2(x)=2|K|$.

By the previous steps, there are exactly two points $p,q\in C$ so that $p+x,q+x\in C$ and the segments $[p,p+x],[q,q+x]$ are contained in $K$. We may assume that $p,p+x,q+x,q$ are ordered clockwise along $C$. The heights of the points in $\sph_K$ projecting over $x$ are given by twice the areas of the sets $A$ and $B$, where $A$ is determined by the portion of $C$ from $p$ to $p+x$ and the segment $[p+x,p]$, and $B$ is determined by the portion of $C$ from $q$ to $q+x$ and the segment $[q+x,q]$. Since $A$ is properly contained in $B$ we have $g_2(x)=2|A|<2|B|=g_1(x)$.

In case $K$ is centrally symmetric, the central symmetry maps $p+x$ to $q$ and $q+x$ to $p$ since $[p,p+x]$ and $[q,q+x]$ are the only segments in $K$ of length $|x|$ with boundary points on $C$. Hence $|A|+|B|=|K|$ and so $g_1(x)+g_2(x)=2|K|$.

\emph{Step 4}. The functions $g_1,g_2$ are of class $C^\ell$ in $\intt(K_0)\setminus\{0\}$.

This follows from the implicit function theorem since $\sph_K$ is $C^\ell$ outside the poles.
\end{proof}

\begin{theorem}
Let $K\subset\rr^2$ be a convex body of class $C^2_+$. Then $\sph_K$ is of class $C^2$ around the poles.
\end{theorem}

\begin{proof}
We consider a horizontal lifting $\Ga=(x,y,t)$ of a clockwise arc-length parametrization $\ga$ of $\ptl K$. Then a parameterization of $\sph_K$ is given by $(\mathbf{x},\mathbf{y},\mathbf{t})(u,v)=\ell_{-\Ga(v)} (\Ga(u+v))$. This means
\begin{equation}
\label{eq:parSK}
\begin{split}
\mathbf{x}(u,v)&=x(u+v)-x(v),
\\
\mathbf{y}(u,v)&=y(u+v)-y(v),
\\
\mathbf{t}(u,v)&=t(u+v)-t(v)-x(u+v)y(v)+y(u+v)x(v).
\end{split}
\end{equation}
The tangent vectors $\ptl/\ptl u, \ptl/\ptl u$ are the image of $(1,0)$ and $(0,1)$ under the parameterization and are given by
\begin{align*}
\frac{\ptl}{\ptl u}&=\dot{x}(u+v)\,X+\dot{y}(u+v)\,Y.
\\
\frac{\ptl}{\ptl v}&=\big(\dot{x}(u+v)-\dot{x}(v)\big)\,X+\big(\dot{y}(u+v)-\dot{y}(v)\big)\,Y+ h(u,v)\,T,
\end{align*}
where
\begin{equation}
\label{eq:defh}
h(u,v)=2\big(\dot{x}(v)(y(u+v)-y(v))-\dot{y}(v)(x(u+v)-x(v))\big).
\end{equation}
Geometrically, $h(u,v)$ is the scalar product of the position vector $(x(u+v)-x(v),y(u+v)-y(v))$ with $J((\dot{x},\dot{y}))$, that is always negative for $u>0$. A Riemannian unit normal vector $N$ can be easily computed from the expressions of $\ptl/\ptl u$ and $\ptl/\ptl v$ and is given by
\begin{equation}
\label{eq:defN}
N=\frac{h\big(\dot{y}(u+v)X-\dot{x}(u+v)Y\big)+gT}{\big(h^2+g^2\big)^{1/2}},
\end{equation}
where 
\begin{equation}
\label{eq:defg}
g(u,v)=\dot{x}(v)\dot{y}(u+v)-\dot{y}(v)\dot{x}(u+v).
\end{equation}
We have
\[
|N_h|=\frac{|h|}{\big(h^2+g^2\big)^{1/2}}, \qquad \escpr{N,T}=\frac{g}{\big(h^2+g^2\big)^{1/2}}
\]

Let us see that $\sph_K$ is a $C^2$ surface near the south pole $(0,0,0)$. The arguments for the north pole of are similar. To see that $\sph_K$ is $C^1$ near the south pole, it is enough to check that $N$ extends continuously to $u=0$. Let us see that
\begin{equation}
\label{eq:limitT}
\lim_{(u,v)\to (0,v_0)} N(u,v)=-T.
\end{equation}
Since $g<0$, from the expression \eqref{eq:defN} it is enough to prove that
\begin{equation}
\label{eq:limithg}
\lim_{(u,v)\to (0,v_0)} \frac{h}{g}(u,v)=0.
\end{equation}
Since $x$ and $y$ are functions of class $C^2$, we use Taylor expansions around $v$ to get
\begin{align*}
x(u+v)&=x(v)+\dot{x}(v) u+R(u,v)u, &
y(u+v)&=y(v)+\dot{y}(v) u+R(u,v)u, &
\\
\dot{x}(u+v)&=\dot{x}(v)+\ddot{x}(v)u+R(u,v)u,
&
\dot{y}(u+v)&=\dot{y}(v)+\ddot{y}(v)u+R(u,v)u.
\end{align*}
In the above equations {$R$ denotes a continuous functions of $(u,v)$} (depending on the equation) that converges to $0$ when $u\to 0$ independently of $v$. This follows from the integral expression for the reminder in Taylor's expansion. Then we have
\begin{align*}
\lim_{(u,v)\to (0,v_0)}\frac{h}{g}(u,v)&=\lim_{(u,v)\to (0,v_0)}\frac{R(u,v)u}{-\kappa(v)u+R(u,v)u}
\\
&=\lim_{(u,v)\to (0,v_0)}\frac{R(u,v)}{-\kappa(v)+R(u,v)}=0,
\end{align*}
where 
\[
\kappa(v)=\big(\dot{y}\ddot{x}-\dot{x}\ddot{y}\big)(v)
\]
is the (positive) geodesic curvature of $\ga$. This proves \eqref{eq:limithg} and so $\sph_K$ is of class $C^1$ around $(0,0,0)$.

To prove that $\sph_K$ is of class $C^2$ around the origin it is enough to show that the Riemannian second fundamental form of $\sph_K$ converges to $0$ when $(u,v)\to (0,v_0)$. We first compute
\[
\lim_{(u,v)\to (0,v_0)} D_{\ptl/\ptl u} N.
\]
Since
\begin{equation}
\label{eq:DuN}
\begin{split}
D_{\ptl/\ptl u}N&=\frac{\ptl}{\ptl u}\bigg(\frac{h\dot{y}(u+v)}{\sqrt{h^2+g^2}}\bigg)\,X
-\frac{\ptl}{\ptl u}\bigg(\frac{h\dot{x}(u+v)}{\sqrt{h^2+g^2}}\bigg)\,Y
+\frac{g}{\sqrt{h^2+g^2}}\,J(\tfrac{\ptl}{\ptl u})
\\
&+\bigg(\frac{\ptl}{\ptl u}\bigg(\frac{g}{\sqrt{h^2+g^2}}\bigg)+\frac{h}{\sqrt{h^2+g^2}}\bigg)\,T.
\end{split}
\end{equation}
A direct computation taking into account $\tfrac{\ptl h}{\ptl u}=2g$ yields
\begin{align*}
\frac{\ptl}{\ptl u}\bigg(\frac{h}{\sqrt{h^2+g^2}}\bigg)&=\frac{2g^3-gh\tfrac{\ptl g}{\ptl u}}{(h^2+g^2)^{3/2}},
&
\frac{\ptl}{\ptl u}\bigg(\frac{g}{\sqrt{h^2+g^2}}\bigg)&=\frac{h^2\tfrac{\ptl g}{\ptl u}-2g^2h}{(h^2+g^2)^{3/2}}.
\end{align*}
It is straightforward to check from the Taylor expressions that
\[
\lim_{(u,v)\to (0,v_0)}\frac{h}{g^2}(u,v)=\lim_{(u,v)\to (0,v_0)}\frac{-\kappa(v_0)u^2+R(u,v)u^2}{\kappa(v_0)^2u^2+R(u,v)u^2}=\frac{-1}{\kappa(v_0)}.
\]
Then we immediately get, dividing by $-g^3$,
\[
\lim_{(u,v)\to (0,v_0)} \frac{\ptl}{\ptl u}\bigg(\frac{h}{\sqrt{h^2+g^2}}\bigg)=\lim_{(u,v)\to (0,v_0)}\frac{-2+\tfrac{h}{g^2}\tfrac{\ptl g}{\ptl u}}{((\tfrac{h}{g})^2+1)^{3/2}}=-1
\]
and
\[
\lim_{(u,v)\to (0,v_0)} \frac{\ptl}{\ptl u}\bigg(\frac{g}{\sqrt{h^2+g^2}}\bigg)=\lim_{(u,v)\to (0,v_0)}\frac{-\tfrac{h}{g}\tfrac{h}{g^2}\tfrac{\ptl g}{\ptl u}+2\tfrac{h}{g}}{((\tfrac{h}{g })^2+1)^{3/2}}=0.
\]
Taking limits in \eqref{eq:DuN} we get
\[
\lim_{(u,v)\to (0,v_0)}D_{\ptl/\ptl u}N=J(\tfrac{\ptl}{\ptl u})-J(\tfrac{\ptl}{\ptl u})+0=0.
\]

We complete $\tfrac{\ptl}{\ptl v}$ to an orthonormal basis of the tangent plane by adding the vector \[
E=\frac{\tfrac{\ptl}{\ptl v}-\escpr{\tfrac{\ptl}{\ptl u},\frac{\ptl}{\ptl v}}\tfrac{\ptl}{\ptl u}}{(1-\escpr{\tfrac{\ptl}{\ptl u},\tfrac{\ptl}{\ptl v}}^2)^{1/2}}.
\] 
Since $\lim_{(u,v)\to (0,v_0)}\tfrac{\ptl}{\ptl v}=0$, we have
\begin{align*}
\lim_{(u,v)\to (0,v_0)}D_EN&=\lim_{(u,v)\to (0,v_0)}D_{\ptl/\ptl v}N
\\
&=\lim_{(u,v)\to (0,v_0)}\bigg(- \frac{\ptl}{\ptl v}\bigg(\frac{h}{\sqrt{h^2+g^2}}\bigg) J(\tfrac{\ptl}{\ptl u})+\frac{\ptl}{\ptl v}\bigg(\frac{g}{(h^2+g^2)^{1/2}}\bigg)\bigg).
\end{align*}
A computation shows that
\[
\frac{\ptl}{\ptl v}\bigg(\frac{h}{\sqrt{h^2+g^2}}\bigg)=\frac{g^2\tfrac{\ptl h}{\ptl v}-gh\tfrac{\ptl g}{\ptl v}}{(h^2+g^2)^{3/2}}, \qquad
\frac{\ptl }{\ptl v}\bigg(\frac{g}{\sqrt{h^2+g^2}}\bigg)=\frac{h^2\tfrac{\ptl g}{\ptl v}-gh\tfrac{\ptl h}{\ptl v}}{(h^2+g^2)^{3/2}}.
\]
We trivially have
\[
\lim_{(u,v)\to (0,v_0)}\frac{\ptl h}{\ptl v}(u,v)=\lim_{(u,v)\to (0,v_0)}\frac{\ptl g}{\ptl v}(u,v)=0.
\]
Hence
\[
\lim_{(u,v)\to (0,v_0)}\frac{\ptl }{\ptl v}\bigg(\frac{g}{\sqrt{h^2+g^2}}\bigg)
=\lim_{(u,v)\to (0,v_0)}\frac{-\tfrac{h}{g}\tfrac{h}{g^2}\tfrac{\ptl g}{\ptl v}+\tfrac{h}{g^2}\tfrac{\ptl h}{\ptl v}}{((\tfrac{h}{g})^2+1)^{3/2}}=0.
\]
On the other hand
\[
\lim_{(u,v)\to (0,v_0)}\frac{\ptl}{\ptl v}\bigg(\frac{h}{\sqrt{h^2+g^2}}\bigg)
=\lim_{(u,v)\to (0,v_0)}\frac{-\tfrac{1}{g}\tfrac{\ptl h}{\ptl v}+\tfrac{h}{g^2}\tfrac{\ptl g}{\ptl v}}{(h^2+g^2)^{3/2}}=0.
\]
This equality holds from the Taylor expansions since
\[
\lim_{(u,v)\to (0,v_0)}\frac{1}{g}\frac{\ptl h}{\ptl v}(u,v)=\lim_{(u,v)\to (0,v_0)}\frac{R(u,v)u}{-\kappa(v)u+R(u,v)u}=0.
\]
So we conclude that $\lim_{(u,v)\to (0,v_0)}D_EN=0$.
\end{proof}

\section{Minimization property of the Pansu-Wulff shapes}
\label{sec:calibration}

We prove in this section a minimization property satisfied by the balls $\bb_K$. Let $K$ be a convex body containing $0$ in its interior. We assume that $K$ is of class $C^\ell_+$, with $\ell\ge 2$.

\begin{remark}
Existence of isoperimetric regions in Carnot and nilpotent groups endowed with a sub-Finsler norm is proved in \cite{pozuelo-nilpotent}. In the Heisenberg group $\hh^1$ with a sub-Finsler norm this is done in \cite[Thm.~3.1]{monti-finsler}. Proofs are based on Leonardi-Rigot's paper \cite{MR2000099}.
\end{remark}

%


\begin{definition}
\label{def:g}
Given $\sph_K$, we let $g:K_0\to\rr$ be the function $g(x)=(g_1(x)+g_2(x))/2$,  where $g_1$ and $g_2$ are the functions obtained in Theorem~\ref{thm:graph}. 

We also introduce the notation $\mathbb{S}_K^+: =\mathbb{S}_K\cap\{(x,t) : t\geq g(x)\}$, $\mathbb{S}_K^-:=\mathbb{S}\cap\{(x,t) : t\leq g(x)\}$ and $D_0=\{(x,g(x)) : x\in K_0\}$.
\end{definition}

\begin{theorem}
\label{cal:1}
Let $\norm{\cdot}_K$ be the norm associated to a convex body $K\subset\rr^2$ of class $C^\ell_+$, with $\ell\ge 2$. Let $r>0$ and $h:rK_0\to\rr$ a $C^0$ function. Consider a subset $E\subset\hh^1$ with finite volume and finite $K$-perimeter such that
\begin{equation*}
\text{\emph{graph}}(h)\subseteq E\subset rK_0\times\mathbb{R}.
\end{equation*} 
Then
\begin{equation}\label{thm}
|\partial E|_K\geq |\partial\mathbb{B}_E|_K,
\end{equation}
where $\mathbb{B}_E$ is the Wulff shape in $(\hh^1,\norm{\cdot}_K)$ with $|E|=|\mathbb{B}_E|$. 
\end{theorem}	

\begin{proof}
Let $g_r:rK_0\to\rr$ the function defined by $g_r(x)=r^2g(\tfrac{1}{r}x)$, where $g$ is the function in Definition~\ref{def:g}. Let $D$ be the graph of $g_r$. We know that $D$ divides the Wulff shape $r\sph_K$ into two parts $r\sph_K^+$ and $r\sph_K^-$. Let $W^+$ and $W^-$ the vector fields in $rK_0\times \rr\setminus L$ defined by translating vertically the vector fields
\[
\pi_K(\nu_0)\big|_{r\sph_K^+}, \quad \pi_K(\nu_0)\big|_{r\sph_K^-},
\]
respectively. Here $\nu_0$ is the horizontal unit normal to $\sph_K$.

As a first step in the proof we are going to show that if $F\subset rK_0\times\rr$ is a set of finite volume and $K$-perimeter so that $\text{rel int}(D)\subset\intt(F)$, then the inequality
\begin{equation}\label{Sus0}
\tfrac{1}{r}|F|\leq \int_{D}\langle W^+-W^-,N_D\rangle dD+|\partial F|_K
\end{equation}
holds, where $N_D$ is the Riemannian normal pointing down and $dD$ is the Riemannian measure of $D$. Equality holds in \eqref{Sus0} if and only if $W^+=\pi_K(\nu_h)$ $|\ptl_KF|$-a.e. on $F^+=F\cap\{t\ge g_r\}$ and $W^-=\pi_K(\nu_h)$ $|\ptl_KF|$-a.e. on $F^-=F\cap\{t\le g_r\}$. Here $\nu_h$ is the horizontal unit normal to $F$.

To prove \eqref{Sus0} we consider two families of functions. For $0<\eps<1$ we consider smooth functions $\varphi_{\eps}$, depending on the Riemannian distance to the vertical axis $L=\{x=y=0\}$, so that $0\le\varphi_\eps\le 1$ and
\begin{align*}
\varphi_{\eps}(p)= 0, \qquad & d(p,L)\le\eps^2, \\
\varphi_{\eps}(p)= 1, \qquad & d(p,L)\ge\eps, \\
|\nabla\varphi_{\eps}(p)|\le 2/\eps, \qquad& \eps^2\le d(p,L)\le\eps.
\end{align*}
Again for $0<\eps<1$ we consider smooth functions $\psi_{\eps}$, depending on the Riemannian distance to the Euclidean hyperplane $\Pi_{0}=\{t=0\}$, so
that $0\le\psi_\eps\le 1$ and
\begin{align*} 
\psi_{\eps}(p)=1, \qquad & d(p,\Pi_{0})\le \eps^{-1/2}, \\
\psi_{\eps}(p)=0, \qquad & d(p,\Pi_{0})\ge \eps^{-1/2}+1, \\
|\nabla\psi_{\eps}(p)|\le 2, \qquad&\eps^{-1/2}\le d(p,\Pi_{0})\le
\eps^{-1/2}+1.
\end{align*}
%
For any $\eps>0$, the vector field $\varphi_\epsilon\psi_\varepsilon W$ has compact support.

It is easy to prove that  $F^{+}$ and $F^-$ have finite $K$-perimeter. Since $F^+$ has also finite (sub-Riemannian) perimeter, applying the Divergence Theorem to $F^+$ and the horizontal vector field $\varphi_\epsilon\psi_\varepsilon W^+$, we have
\begin{equation}\label{Calproof:1}
\begin{split}
\int_{F^+}\text{div}(\varphi_\epsilon\psi_\varepsilon W^+)d\hh^1&=\int_{D}\langle\varphi_\epsilon\psi_\varepsilon W^+,N_D\rangle dD
\\
&+\int_{\{t>g_r\}}\langle\varphi_\epsilon\psi_\varepsilon W^+,\nu_h\rangle d|\partial F|.
\end{split}
\end{equation}
Where $N_D$ is the Riemannian unit normal to $D$ pointing into $F^-$, $dD$ is the Riemannian area element on $D$, and $\nu_h$ is the outer horizontal unit normal to $F$.

We take limits in the left hand side of Equation \eqref{Calproof:1} when $\varepsilon\to0$.  We write
\begin{equation}\label{cuenta1}
\int_{F^+}\text{div}(\varphi_\epsilon\psi_\varepsilon W^+)d\hh^1=\int_{F^+}\varphi_\epsilon\psi_\varepsilon \divv W^+ d\hh^1+\int_{F^+}\langle\nabla(\varphi_\epsilon\psi_\varepsilon), W^+\rangle d\hh^1.
\end{equation}
Since $\langle\varphi_\varepsilon\nabla\psi_\varepsilon,W^+\rangle$ is bounded and converges pointwise to $0$, and
\[
\int_{F^+}\langle\psi_\varepsilon\nabla\varphi_\varepsilon,W^+\rangle\leq\int_{\{(x,t): \varepsilon^2<|x|<\varepsilon, \ 0<t<\varepsilon^{-1/2}+1 \}}\psi_\varepsilon|\nabla\varphi_\epsilon| d\hh^1,
\]
we have
\begin{equation}\label{cuenta2}
\lim_{\varepsilon\to 0}\int_{F^+}\langle\nabla(\varphi_\epsilon\psi_\varepsilon),W^+\rangle d\hh^1=0.
\end{equation}
On the other hand, $\divv W^+=\tfrac{1}{r}$, the mean curvature of $r\bb_K$. We consider the orthonormal  vectors $Z=-J(\nu_h)$, $E=\escpr{N,T}\,\nu_h-|\nu_h|\,T$ and $N$, globally defined on $(rK_0\times\rr)\setminus L$ by vertical translations. We know from Lemma~\ref{lem:defH} that
\[
\escpr{D_Z W^+,Z}=\tfrac{1}{r}, \qquad \escpr{D_E W^+ ,E}=2\,\escpr{N,T}|N_h|\,\escpr{W^+,J(\nu_h)}.
\]
It remains to compute $\escpr{D_N W^+,N}$. We express $N=\la E+\mu T$ as a linear combination of $E$ and $T$, where $\la=|N_h|/\escpr{N,T}$, $\mu=1/\escpr{N,T}$. Observe that $\escpr{N,T}\neq 0$ on $\intt(K_0)$ since $r\sph_K^+$ is a $t$-graph. So we have
\begin{align*}
\escpr{D_N W^+,N}&=\la\escpr{D_E W^+,N}+\mu\escpr{D_T W^+,N}
\\
&=\la^2\escpr{D_EW^+,E}+\la\mu\escpr{D_EW^+,T}+\mu\escpr{J(W^+),N_h}
\\
&=\la^2\escpr{D_EW^+,E}-\la\mu\escpr{N,T}\escpr{W^+,J(\nu_h)}-\mu|N_h|\escpr{W^+,J(\nu_h)}
\\
&=\bigg(\frac{|N_h|}{\escpr{N,T}}\bigg)^2\escpr{D_E W^+,E}-\frac{1}{\escpr{N,T}^2}\,\escpr{D_EW^+,E}
\\
&=\escpr{D_E W^+,E},
\end{align*}
where we have used that $D_T W^+=J(W^+)$ since $W^+$ is a linear combination of $W^+,Y$ multiplied by functions that do not depend on $t$.
%
Hence
\begin{equation*}
\divv W^+=\escpr{D_Z W^+,Z}+\escpr{D_E W^+,E}+\escpr{D_N W^+,N}=\frac{1}{r}
\end{equation*}
on $\intt(K_0)$. Since $\varphi_\epsilon\psi_\varepsilon\divv W^+$ is uniformly bounded, $F^+$ has finite volume and $\lim_{\varepsilon\to 0}\varphi_\epsilon\psi_\varepsilon=1$, we can apply Lebesgue's Dominated Convergence Theorem to get
\begin{equation}\label{cuenta3}
\lim_{\varepsilon\to 0}\int_{F^+}\varphi_\epsilon\psi_\varepsilon\divv W^+ d\hh^1=\tfrac{1}{r}|F^+|.
\end{equation}
So we get from \eqref{cuenta1}, \eqref{cuenta2} and \eqref{cuenta3}
\begin{equation}
\label{eq:limitdiv}
\lim_{\eps\to 0}\int_{F^+}\divv(\varphi_\eps\psi_\eps W^+)\,d\hh^1=\tfrac{1}{r}\,|F^+|.
\end{equation}

Now we treat the remainings terms in \eqref{Calproof:1}. Using the representation of perimeter obtained in \eqref{eq:rnfinal} for sets of finite $K$-perimeter sets we have
\begin{equation}\label{cuenta4}
\int_{\{t>g_r\}}\langle W^+,\nu_h\rangle d|\partial F|\leq\int_{\{t>g_r\}}\|\nu_h\|_* d|\partial F|=|\partial F^+|_K,
\end{equation}
with equality if and only if $W^+=\pi(\nu_h)$ $|\partial F|$-a.e. on $\{t>g_r\}$. From equations \eqref{eq:limitdiv} and \eqref{cuenta4}, taking limits in Equation \eqref{Calproof:1} when $\varepsilon\to0$, 
\begin{equation}\label{Calproof:2}
\tfrac{1}{r}|F^+|\leq\int_{D}\langle W^+,N_D\rangle dD+|\partial F^+|_K,
\end{equation}
with equality if and only if $W^+=\pi(\nu_h)$ $|\partial F|$-a.e. on $\partial F\cap\{t>g_r\}$.

We consider now the foliation of $rK_0\times\rr$ by vertical translations of $r\mathbb{S}_K^-$.  Reasoning as in the previous case we get
\begin{equation}\label{Calproof:3}
\tfrac{1}{r}|F^-|\leq -\int_{D}\langle W^-,N_D\rangle dD+|\partial F^-|_K.
\end{equation}
with equality if and only if $W^-=\pi(\nu_h)$ $|\partial F|$-a.e. on $\partial \{t<g_r\}$.
Hence, adding \eqref{Calproof:2} and \eqref{Calproof:3}, and taking into account $|\partial F|_K(\hh^1\setminus D)= |\partial F|_K$ and that $F\cap D$ does not contribute to the volume of $F$, we get
\begin{equation*}
\tfrac{1}{r}|F|\leq \int_{D}\langle W^+-W^-,N_D\rangle dD+|\partial F|_K,
\end{equation*}
and so \eqref{Sus0} holds, with equality if and only if equalities \eqref{Calproof:2} and \eqref{Calproof:3} hold. This completes the first part of the proof.

Recall that $h:rK_0\to\rr$ is a function so that $D=\text{graph}(h)\subset E$. We take two values $t_m<t_M$ such that
\[
h+t_m< g_r< h+t_M.
\]
We apply inequality \eqref{Sus0} to the set $B=B^-\cup B^0\cup B^+$, where
\begin{itemize}
\item $B^0=\{(x,t) : x\in rK_0, |t-g_r|\le (t_M-t_m)/2\}$, 
\item $B^+=r\bb_K^++(0,(t_M-t_m)/2)$,
\item $B^-=r\bb_K^--(0,(t_M-t_m)/2)$.
\end{itemize}

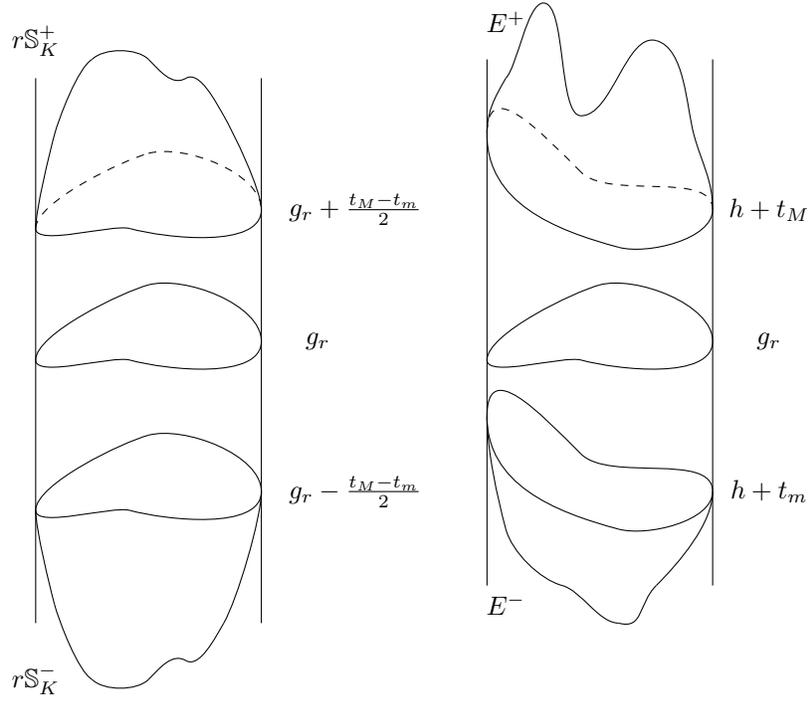
\begin{figure}[h]
\begin{tikzpicture}
	\begin{pgfonlayer}{nodelayer}
		\node [style=none] (0) at (-4, 2.75) {};
		\node [style=none] (1) at (-4, -4.5) {};
		\node [style=none] (2) at (-1, 2.75) {};
		\node [style=none] (3) at (-1, -4.5) {};
		\node [style=none] (4) at (2, 3) {};
		\node [style=none] (6) at (5, 3) {};
		\node [style=none] (8) at (2, -4) {};
		\node [style=none] (9) at (5, -4) {};
		\node [style=none] (14) at (-4, 0.75) {};
		\node [style=none] (16) at (-1, 1) {};
		\node [style=none] (17) at (-2.75, 0.75) {};
		\node [style=none] (18) at (2, 2) {};
		\node [style=none] (19) at (5, 1) {};
		\node [style=none] (21) at (3.75, 0.5) {};
		\node [style=none] (22) at (2, -1.75) {};
		\node [style=none] (23) at (5, -2.75) {};
		\node [style=none] (24) at (3.25, -2.25) {};
		\node [style=none] (25) at (3.75, -3.25) {};
		\node [style=none] (26) at (0.25, 1) {$g_r+\frac{t_M-t_m}{2}$};
		\node [style=none] (28) at (-4, -3) {};
		\node [style=none] (29) at (-2.5, -2) {};
		\node [style=none] (30) at (-1, -2.75) {};
		\node [style=none] (31) at (-2.75, -3) {};
		\node [style=none] (32) at (0.25, -2.75) {$g_r-\frac{t_M-t_m}{2}$};
		\node [style=none] (37) at (-0.25, -0.75) {$g_r$};
		\node [style=none] (38) at (5.75, 1) {$h+t_M$};
		\node [style=none] (39) at (5.75, -0.75) {$g_r$};
		\node [style=none] (40) at (5.75, -2.75) {$h+t_m$};
		\node [style=none] (41) at (-3.75, 2) {};
		\node [style=none] (42) at (-3.25, 3) {};
		\node [style=none] (43) at (-2.5, 3) {};
		\node [style=none] (44) at (-2, 2.75) {};
		\node [style=none] (45) at (-1.5, 2.25) {};
		\node [style=none] (46) at (-4, -3) {};
		\node [style=none] (47) at (-1, -2.75) {};
		\node [style=none] (48) at (-3.75, -4.25) {};
		\node [style=none] (49) at (-3.25, -5.25) {};
		\node [style=none] (50) at (-2.5, -5.25) {};
		\node [style=none] (51) at (-2, -5) {};
		\node [style=none] (52) at (-1.5, -4.5) {};
		\node [style=none] (53) at (2.25, 2.75) {};
		\node [style=none] (54) at (2.75, 3.75) {};
		\node [style=none] (55) at (3.25, 2.25) {};
		\node [style=none] (56) at (4.25, 3.25) {};
		\node [style=none] (57) at (4.75, 2) {};
		\node [style=none] (58) at (2.25, -3.25) {};
		\node [style=none] (59) at (3, -4) {};
		\node [style=none] (60) at (3.75, -4.5) {};
		\node [style=none] (61) at (4.25, -4) {};
		\node [style=none] (63) at (-4, 3.25) {$r\mathbb{S}_K^+$};
		\node [style=none] (64) at (-4, -5.25) {$r\mathbb{S}_K^-$};
		\node [style=none] (65) at (2.25, 3.5) {$E^+$};
		\node [style=none] (66) at (2.25, -4.25) {$E^-$};
		\node [style=none] (67) at (-4, -1) {};
		\node [style=none] (68) at (-2.5, 0) {};
		\node [style=none] (69) at (-1, -0.75) {};
		\node [style=none] (70) at (-2.75, -1) {};
		\node [style=none] (71) at (2, -1) {};
		\node [style=none] (72) at (3.5, 0) {};
		\node [style=none] (73) at (5, -0.75) {};
		\node [style=none] (74) at (3.25, -1) {};
		\node [style=none] (75) at (-4, 0.75) {};
		\node [style=none] (76) at (-2.5, 1.75) {};
		\node [style=none] (77) at (-1, 1) {};
		\node [style=none] (78) at (2, 2) {};
		\node [style=none] (79) at (5, 1) {};
		\node [style=none] (80) at (3.25, 1.5) {};
	\end{pgfonlayer}
	\begin{pgfonlayer}{edgelayer}
		\draw (0.center) to (1.center);
		\draw (2.center) to (3.center);
		\draw (4.center) to (8.center);
		\draw (6.center) to (9.center);
		\draw [in=-15, out=-90, looseness=0.75] (16.center) to (17.center);
		\draw [in=-90, out=165, looseness=0.50] (17.center) to (14.center);
		\draw [in=165, out=-90] (18.center) to (21.center);
		\draw [in=-90, out=-15, looseness=0.75] (21.center) to (19.center);
		\draw [in=165, out=-90] (22.center) to (25.center);
		\draw [in=-90, out=-15, looseness=0.75] (25.center) to (23.center);
		\draw [in=-45, out=90, looseness=0.75] (23.center) to (24.center);
		\draw [in=90, out=135, looseness=1.50] (24.center) to (22.center);
		\draw [in=-165, out=90, looseness=0.50] (28.center) to (29.center);
		\draw [in=90, out=15, looseness=0.75] (29.center) to (30.center);
		\draw [in=-15, out=-90, looseness=0.75] (30.center) to (31.center);
		\draw [in=-90, out=165, looseness=0.50] (31.center) to (28.center);
		\draw [in=-105, out=90, looseness=0.50] (14.center) to (41.center);
		\draw [in=-150, out=75, looseness=0.50] (41.center) to (42.center);
		\draw [in=135, out=45, looseness=0.75] (42.center) to (43.center);
		\draw [in=-150, out=-45] (43.center) to (44.center);
		\draw [in=120, out=30, looseness=0.75] (44.center) to (45.center);
		\draw [in=90, out=-60, looseness=0.50] (45.center) to (16.center);
		\draw [in=105, out=-90, looseness=0.50] (46.center) to (48.center);
		\draw [in=150, out=-75, looseness=0.50] (48.center) to (49.center);
		\draw [in=-135, out=-45, looseness=0.75] (49.center) to (50.center);
		\draw [in=150, out=45] (50.center) to (51.center);
		\draw [in=-120, out=-30, looseness=0.75] (51.center) to (52.center);
		\draw [in=-90, out=60, looseness=0.50] (52.center) to (47.center);
		\draw [in=240, out=90] (18.center) to (53.center);
		\draw [in=-180, out=45, looseness=0.50] (53.center) to (54.center);
		\draw [in=180, out=0, looseness=0.50] (54.center) to (55.center);
		\draw [in=165, out=0, looseness=0.75] (55.center) to (56.center);
		\draw [in=105, out=-15, looseness=0.75] (56.center) to (57.center);
		\draw [in=90, out=-75] (57.center) to (19.center);
		\draw [in=105, out=-90, looseness=0.75] (22.center) to (58.center);
		\draw [in=165, out=-75, looseness=0.75] (58.center) to (59.center);
		\draw [in=180, out=-15, looseness=0.75] (59.center) to (60.center);
		\draw [in=225, out=-15, looseness=1.25] (60.center) to (61.center);
		\draw [in=-90, out=45, looseness=0.75] (61.center) to (23.center);
		\draw [in=-165, out=90, looseness=0.50] (67.center) to (68.center);
		\draw [in=90, out=15, looseness=0.75] (68.center) to (69.center);
		\draw [in=-15, out=-90, looseness=0.75] (69.center) to (70.center);
		\draw [in=-90, out=165, looseness=0.50] (70.center) to (67.center);
		\draw [in=-165, out=90, looseness=0.50] (71.center) to (72.center);
		\draw [in=90, out=15, looseness=0.75] (72.center) to (73.center);
		\draw [in=-15, out=-90, looseness=0.75] (73.center) to (74.center);
		\draw [in=-90, out=165, looseness=0.50] (74.center) to (71.center);
		\draw [style=dashed, in=-165, out=90, looseness=0.50] (75.center) to (76.center);
		\draw [style=dashed, in=90, out=15, looseness=0.75] (76.center) to (77.center);
		\draw [style=dashed, in=-45, out=90, looseness=0.75] (79.center) to (80.center);
		\draw [style=dashed, in=90, out=135, looseness=1.50] (80.center) to (78.center);
	\end{pgfonlayer}
\end{tikzpicture}
\caption{Geometric construction in the proof of Theorem~\ref{cal:1}}
\end{figure}

By construction, $D=\text{graph}(g_r)\subset B^0$. Since the lateral boundary of $B^0$ is contained in $\ptl(rK_0\times\rr)$ and the outer unit normal to $\ptl(rK_0\times\rr)$ coincides with $W^+$ and $W^-$, the lateral $K$-boundary area of $B^0$ is equal to
\[
(t_M-t_m)\int_{\ptl (rK_0)} \norm{\nu_0}_* d(\ptl(rK_0)),
\]
where $d(\ptl(rK_0))$ is the Riemannian length element of the $C^1$ curve $\ptl(rK_0)$. Hence we get
\begin{equation*}
|\partial B|_K=(t_M-t_m)\int_{\partial (rK_0)}\|\nu_0\|_*d(\partial (rK_0))+|\partial (r\bb_K)|_K.
\end{equation*}
On the other hand, since
\[
|B|=|r\bb_K|+|rK_0|(t_M-t_m),
\]
we obtain 
\begin{equation}
\label{eq:Sus1}
\begin{split}
\tfrac{1}{r}(|r\bb_K|&+|rK_0|(t_M-t_m))= \int_{D}\langle W^+-W^-,N_D\rangle dD
\\
& +(t_M-t_m)\int_{\partial (rK_0)}\|\nu_0\|_*d\partial (rK_0)+|\partial (r\bb_K)|_K.
\end{split}
\end{equation}


Now we apply \eqref{Sus0} to the set $\mathbb{E}$  consisting on the union of $E^+=E\cap\{t\ge h\}$ translated by the vector $(0,t_M)$, $E^-=E\cap\{t\le h\}$ translated by the vector $(0,t_m)$ and the vertical filling in between the two sets. We reason as before to get
\begin{equation}
\label{eq:Sus2}
\begin{split}
\tfrac{1}{r}(|E|+|rK_0|(t_M-t_m))&\leq \int_{D}\langle W^+-W^-,N_D\rangle dD
\\ 
&+(t_M-t_m)\int_{\partial D}\|\nu_0\|_*d\partial D_0+|\partial E|_K.
\end{split}
\end{equation}
From \eqref{eq:Sus1} and \eqref{eq:Sus2} we get  
\begin{equation*}
|\partial E|_K\geq |\partial (r\mathbb{B}_K)|_K+\tfrac{1}{r}(|E|-|r\mathbb{B}|).
\end{equation*}

Let $f(\rho)=|\partial(\rho\mathbb{B}_K)|_K+\tfrac{1}{\rho} (|E|- |\rho\mathbb{B}|)$.  Since $\rho\bb_K$ has mean curvature $\tfrac{1}{\rho}$, 
Theorem \ref{prop:varform} guarantees that the Wulff shape $\rho\bb_K$ is a critical point of $A-\tfrac{1}{\rho}|\cdot|$ for any variation. Therefore $|\partial(\rho\mathbb{B}_K)|_K'-\tfrac{1}{\rho} | \rho\bb_K|'=0$ where primes indicates the derivative with respect to $\rho$. Hence we have
\[
f'(\rho)=-\tfrac{1}{\rho^2}(|E|-|\rho\bb_K|).
\]
So the only critical point of $f$ corresponds to the value $\rho_0$ so that $|\rho_0\bb_K|=|E|$. Since the function $\rho\mapsto |\rho B_K|$ is strictly increasing and takes its values in $(0,+\infty)$, we obtain that $f(\rho)$ is a convex function with a unique minimum at $\rho_0$. Hence we obtain
\begin{equation*}
|\partial E|_K\geq f(r)\geq f(\rho_0)= |\partial(r_0\bb_K)|_K, 
\end{equation*}
which implies \eqref{thm}.
\end{proof}	


\end{document}